\documentclass[11pt,a4paper]{article}%{elsarticle}%11 pt for EJOR
\usepackage[english]{babel}
\usepackage[T1]{fontenc}
\usepackage[utf8]{inputenc}

\usepackage{bm}
\usepackage{lmodern}
\usepackage[normalem]{ulem}

 \usepackage[top=2.25cm, bottom=2.75cm, left=2.75cm, right=2.75cm]{geometry}
\usepackage{latexsym}
\usepackage{pdflscape}
\usepackage{fancyhdr}
\usepackage{setspace}

\usepackage{graphicx}
\usepackage[dvipsnames]{xcolor}
\usepackage{epstopdf}
\usepackage{epsf}
\usepackage{eepic}
\usepackage{tikz}
\usetikzlibrary{shadings,shapes,arrows,calc,positioning,shadings,decorations.markings, patterns}

\usepackage[numbers]{natbib}

\usepackage{amsmath,amssymb,amsfonts,amscd}
\usepackage{mathabx}

\usepackage{enumitem}
\usepackage{caption}
\usepackage[format=hang,margin=10pt]{subcaption}
\usepackage{listings}
\usepackage{varwidth}

\usepackage{algorithm}
\usepackage{algpseudocode}
\usepackage[numbered,framed]{matlab-prettifier}
\usepackage{amsthm}
\usepackage[linguistics]{forest}
\forestset{
nice empty nodes/.style={
    for tree={l sep*=2},
    delay={where content={}{shape=coordinate}{}}
},
}

\usepackage{changepage}
\usepackage{xifthen}
\usepackage{todonotes}

\usepackage{hyperref}
\usepackage{cleveref}

\numberwithin{equation}{section}
\setlist[itemize,1]{label=$\bullet$}
\setlist[itemize,2]{label=$\triangleleft$}
\setlist[enumerate,1]{label=(\roman*)}
\setlist[enumerate,2]{label=(\arabic*)}

\hypersetup{
    breaklinks=true,
    unicode=true,
    pdfpagelayout=OneColumn,
    bookmarksnumbered=true,
    bookmarksopen=true,
    bookmarksopenlevel=0,
    pdfborder={0 0 0},  
    colorlinks=true,
    linkcolor=Cerulean,    
    citecolor=Cerulean,
}

\algnewcommand\algorithmicinput{\textbf{Input:}}
\algnewcommand\AlgInput{\item[\algorithmicinput]}
\algnewcommand\algorithmicoutput{\textbf{Output:}}
\algnewcommand\AlgOutput{\item[\algorithmicoutput]}

\lstMakeShortInline"
\newtheoremstyle{dotless}{}{}{\itshape}{}{\bfseries}{}{ }{}
\newtheoremstyle{no-italic}{}{}{}{}{\bfseries}{}{ }{}
\theoremstyle{dotless}
\newtheorem{Theorem}{Theorem}[section]
\newtheorem{Example}[Theorem]{Example}
\newtheorem{Lemma}[Theorem]{Lemma}
\newtheorem{Definition}[Theorem]{Definition}
\newtheorem{Assumption}{Assumption}

\newtheorem{Remark}[Theorem]{Remark}
\newtheorem{Proposition}[Theorem]{Proposition}

\newtheorem{Test Instance}[Theorem]{Test Instance}
\theoremstyle{no-italic}

\newcommand*{\R}{\mathbb{R}}

\newcommand*{\N}{\mathbb{N}}

\DeclareMathOperator{\cl}{cl}			% closure
\DeclareMathOperator{\wargmin}{wargmin}	% wargmin

\usepackage{tikz}

\def\R{{\mathbb R}}

\def\Min{\textup{Min}}

\def\Int{\textup{int }}

\def\cl{\textup{cl}}

\def\gph{\textup{gph }}

\title{Set-based Robust Optimization of\\ Uncertain Multiobjective Problems via\\ Epigraphical Reformulations}
\author{Gabriele Eichfelder\thanks{(Corresponding author) Institute of Mathematics, Technische Universität Ilmenau, Po 10 05 65, D-98684 Ilmenau, Germany,
{\texttt{gabriele.eichfelder@tu-ilmenau.de}}}\ \   and   Ernest Quintana\thanks{Institute of Mathematics, Technische Universität Ilmenau, Po 10 05 65, D-98684 Ilmenau, Germany, {\texttt{ernestqtptf@icloud.com}}}}
\date{\today}

\begin{document}

%%%%%%%%%%%%%%%%
%% Title page %%
%%%%%%%%%%%%%%%%
\maketitle

%%%%%%%%%%%%%%
%% Abstract %%
%%%%%%%%%%%%%%
\begin{abstract}
In this paper, we study a method for finding 
robust solutions to multiobjective optimization problems under uncertainty. We follow the set-based minmax approach for handling the uncertainties which leads to a certain set optimization problem with the strict upper type set relation. We introduce, under some assumptions, a reformulation using instead the strict lower type set relation without sacrificing the compactness property of the image sets. This allows to apply vectorization results to characterize the optimal solutions of these set optimization problems as optimal solutions of a multiobjective optimization problem.  We end up with multiobjective semi-infinite problems which can then be studied with classical techniques from the literature.
\end{abstract}

%%%%%%%%%%%%%%%%%%%%%%%%%%%%%%%%%%
%% Key words and classification %%
%%%%%%%%%%%%%%%%%%%%%%%%%%%%%%%%%%

\vspace{2ex} \noindent {\small\textbf{Mathematics subject
classifications (MSC 2010):}}
90C29, %: Multi-objective and goal programming
90C31, % Sensitivity, stability, parametric optimization
90C47, % Minimax problems
90C34. % Semi-infinite programming
 %: Methods involving duality
%%%%%%%%%%%%%
%% Content %%
%%%%%%%%%%%%%

\section{Introduction}
\label{section:intro}
In real-world applications of optimization, it is usually the case that the problem to solve involves the minimization of different competing objectives simultaneously. Even more, the data defining the problem is most of the time obtained through methods that have known sources of errors, like the results of experiments or the measurements of some technical equipment. As a consequence, the problem is, in these cases, not completely available to the decision maker at optimization time. Traditionally, the fields of multiobjective and robust optimization have dealt with these two difficulties independently, namely the minimization of multiple objectives  and the uncertainty in the data, respectively. However, because of practical relevance, the combination of techniques from both areas have received a lot of attention during the last decade.

In robust optimization, the uncertainty of the data is modeled by identifying the uncertain problem as a parametric family of (deterministic) optimization problems, where every such problem is associated to one possible value of the uncertain vector of parameters in the model. This space of possible parameters is usually called the uncertainty set and must be specified before any optimization process can take place. Most of the research in this area then focuses on the definition of suitable optimality concepts and algorithms for finding solutions corresponding to those concepts. Typically, the main ingredient in the definition of an optimality concept is the construction of a deterministic optimization problem that represents, according to some specific criteria, the parametric family already modeled. These problems are referred in the literature as robust counterpart problems, and their solutions are called the robust solutions of the uncertain problem (in the sense of the corresponding robust counterpart). Thus, the robust counterpart problem plays a fundamental role as it completely specifies a concept of robust solution. 

For single-objective optimization problems under uncertainty, perhaps the most popular concept is the one coming from the worst case or minmax approach \cite{bentalghaouinemirovski2009}. There, roughly speaking, the robust counterpart problem attempts to minimize the supremum of the objective function over the uncertainty set while satisfying all possible constraints associated to all possible values of the vector of uncertain parameters. Throughout the years, two main extensions of this concept from the scalar case to the multiobjective setting have been proposed in the literature, namely the so called point-based minmax approach \cite{KuroiwaLee2012} and the set-based minmax approach \cite{EhrgottIdeSchobel2014}. In the point-based minmax approach, the robust counterpart problem minimizes the supremums over the uncertainty set of every objective function individually. However, by doing this, the fact that every objective function depends on a common number of uncertain parameters is missed and, as a consequence, the obtained robust solutions may be too conservative to be employed in practice.  For a discussion on the price of robustness see \cite{Schoebel21}. This modelling drawback is somewhat alleviated with the set-based minmax approach. There, the robust counterpart problem is a set-valued optimization problem \cite{KTZ} in which a set-valued mapping associating to every decision a set of possible vector costs is minimized, therefore giving greater flexibility to the decision maker. However, because of the set-valued nature of the robust counterpart in this case, finding robust solutions w.r.t.\ this approach is computationally more challenging. This is precisely the topic of this paper.

Except for the  branch and bound scheme described in \cite{eichfeldernieblingrocktaschel2019} for the specific setting of decision uncertainty, algorithms for uncertain multiobjective optimization are based on some type of scalarization, see for instance  \cite{EhrgottIdeSchobel2014,idekobis2014, IdeKobisKuroiwa2014,jiangcaoxiong2019, schmidtschobelthom2019}.  
The main drawback of the scalarization-based methods   is that, in general, they are not able to recover all the solutions of the set optimization problem. In fact, the $\epsilon$-constraint method, which is known to overcome this difficulty in standard multiobjective optimization, will fail in this setting. Thus, algorithms that are able to deal with this problem are of interest. 

Our main contribution to this aim is a generalization of the epigraphical reformulation known from  single-objective uncertain optimization to the multiobjective setting. This approach consists of replacing the robust counterpart by a family of semi-infinite multiobjective optimization problems. In particular, we show that by doing this the structural difficulties of the set-valued robust counterpart are removed and, as a consequence, in a second step the decision maker can tackle these problems by combining classical techniques from multiobjective optimization and semi-infinite programming. 

 The paper is organized as follows. In Section 2 we collect the preliminaries including the required concepts from multiobjective and from set optimization, and we recall the vectorization results from \cite{EichfelderQuintanaRocktaeschel2022} which are a main ingredient for our approach. Moreover, we formulate the uncertain multiobjective optimization problem. In Section 3 we find an equivalence between the different types of set relations which allows to make use of the mentioned vectorization results. As a consequence, we can formulate in Section 4 the epigraphical reformulation and we can provide  results on the relation between the optimal solution sets of the reformulation and the uncertain multiobjective problem. Under some additional assumptions we even guarantee exactness of the approximation. We then conclude the paper in Section 5 with some final remarks.

\section{Preliminaries}
\label{section:prelims}

We begin this section by introducing the notation that will be used throughout the paper. We write $[p]$ for $\{1,\ldots,p\}$ for any $p\in\N$. For a set $A\subseteq\R^m$,  $|A|$,  $\cl\; A,\ \Int A\;,\ A^c  $ denote   the cardinality, the closure, the  interior and the complement  of $A$, respectively. We denote by $e$ the all-one vector of appropriate dimensions. %The set 
%$\mathcal{S}_m := \{\lambda \in \R^m_+ \mid \lambda^\top e = 1\}$ denotes the unit simplex and 
$\|\cdot\|$ denotes, unless otherwise stated, the Euclidean norm. Inequalities between vectors, i.e., $\leq$ and $<$, are understood in a component-wise sense, and if  $f: \R^n \rightarrow \R^m$, we denote by $f_j$ the $j^{th}$-component of $f$.

\subsection{Multiobjective Optimization}
We now start by recalling the notion of a multiobjective optimization problem, together with some of its optimal solution concepts. Let a nonempty set $\Omega \subseteq \R^n$ and a vector-valued function $f\colon  \Omega \to  \R^m$ be given. Then, the multiobjective optimization problem associated to this data is defined as  
\begin{equation}\label{eq:vpg}\tag{$\mathcal{MP}$}
\begin{array}{ll}
\min\limits_{x} \; f(x) \\
  \; \textup{s.t.} \;\; x \in \Omega.\\
\end{array}
\end{equation} 
We say that $\bar{x}\in\Omega$ is a weakly efficient solution of $\eqref{eq:vpg}$ if  there exists no $x\in \Omega$ with $ f(x) < f(\bar{x})$  and we denote the set of weakly efficient solutions by $\wargmin \eqref{eq:vpg}$.
We say that $\bar{x}\in\Omega$ is an efficient solution of $\eqref{eq:vpg}$ if there exists no $x\in \Omega$ with $ f(x) \leq f(\bar{x})$ and $  f(x) \neq f(\bar{x})$. Moreovoer, we say that $\bar{x}\in\Omega$ is a strongly efficient solution of $\eqref{eq:vpg}$ if it holds for all $x\in\Omega$ that $f(\bar{x}) \leq  f(x)$.

Closely related, for a set $A \subseteq \R^m,$ we will denote by $\Min(A, \R^m_+)$  the set of all efficient solutions of the multiobjective optimization problem
\begin{equation*}
\begin{array}{ll}
\min\limits_{y} \; y \\
  \; \textup{s.t.} \;\; y \in A.\\
\end{array}
\end{equation*} 

%It is worth mentioning that there are other solution concepts not considered in Definition \ref{def:esolutions vp}, like those of efficiency and proper efficiency, that are more more relevant from the applied point of view  although less rich theoretically. We have thus presented only those concepts that will be useful in our work. 

Moreover, from the practical point of view it is important to consider different notions of approximate solutions to multiobjective problems. For weakly efficient solutions, the following concept was introduced by Kutateladze \cite{Kutateladze1979} and will be often used in our results:

\begin{Definition}
Let $\varepsilon \geq 0$ be given. We say that $\bar{x} \in \Omega$ is an $\varepsilon$-weakly efficient solution of $\eqref{eq:vpg}$ if 
       $$ %\begin{equation}
        \nexists \; x\in \Omega: f(x) < f(\bar{x}) - \varepsilon e.
       $$ %\end{equation}
        The set of $\varepsilon$-weakly efficient solutions is denoted by $\varepsilon\textrm{-}\wargmin \eqref{eq:vpg}.$   
\end{Definition}

\subsection{Set Optimization}

Next we recall a class of problems that generalize multiobjective problems, and that is referred in the literature as set optimization problems. There, we deal with the minimization of set-valued mappings. This class of problems is of particular interest to us because of two main reasons. On the one hand, they provide a unified way to define the solutions of the uncertain multiobjective optimization problem \eqref{ump}, see page \pageref{ump},  that we will treat in this paper. On the other hand, previous results derived for set optimization problems will be a main tool in our work. 

A distinctive feature of set optimization is the use of binary relations in order to compare the images of the set-valued objective mapping. Therefore, we first need to define those that will be used in our context. For these binary relations and the further basic concepts of set optimization we  refer to the recent book \cite{KTZ}.

\begin{Definition}\label{def-set-relation}
The following binary relations are defined on the class of nonempty subsets of $\R^m,$ i.e. for arbitrary nonempty sets $A,B \subseteq \R^m$:
\begin{enumerate}
    \item 
    $A \prec^l B: \Longleftrightarrow B\subseteq A+ \Int\R_+^m,$
    \item 
    $A \preceq^l B: \Longleftrightarrow B\subseteq A+\R_+^m,$
    \item   
    $A \prec^u B: \Longleftrightarrow A\subseteq B - \Int\R_+^m,$
    \item 
    $A \preceq^u B: \Longleftrightarrow A\subseteq B -\R_+^m.$
    
\end{enumerate}
%For $A,B \subseteq \R^m,$ w
We write $A\nprec^l B$ and $A\nprec^u B$ if the inequalities $A\prec^l B$ and $A\prec^u B$ are not satisfied, respectively. In a similar way, we write $A\npreceq^l B$ and $A\npreceq^u B.$
\end{Definition}

The binary relations $\preceq^l, \prec^l, \preceq^u, \prec^u$ in Definition \ref{def-set-relation} are referred in the literature as the lower type, strict lower type, upper type, and strict upper type set relations, respectively. It is well known that $\preceq^l$ and $\preceq^u$ are preorder relations, but that they are not necessarily antisymmetric. Likewise, $\prec^l$ and $\prec^u$ are transitive, but in general fail to be reflexive or antisymmetric.

Let now $\Omega \subseteq \R^n$ be a nonempty closed set and $F\colon  \Omega \rightrightarrows \R^m$ be a given set-valued mapping with $F(x)$ nonempty for all $x\in\Omega$. Then, the set optimization problem corresponding to this data is 
\begin{equation}\label{sp}\tag{$\mathcal{SP}$}
\begin{array}{rl}
&\min\limits_{x} \; F(x) \\
                        &\;\textup{s.t.} \;\; x \in \Omega.\\
\end{array}  
\end{equation} 

\begin{Definition}\label{def:sp solution concept}
Let $\bar{x} \in \Omega$ and let $\star \in \{l,u\}.$ We say that $\bar{x}$ is a $\star$-weakly efficient solution of $\eqref{sp}$ if there exists no $x\in\Omega$ with $F(x) \prec^\star F(\bar{x})$. The set of $\star$-weakly efficient solutions of \eqref{sp} is denoted by $\wargmin^\star \eqref{sp}.$ 
\end{Definition}

In \cite{EichfelderQuintanaRocktaeschel2022}, the authors introduced a vectorization scheme for finding $l$-weakly efficient solutions of \eqref{sp}. There, for $p \in \N,$ the multiobjective optimization problem
\begin{equation}\label{vpp}
\begin{array}{rl}
 & \min\limits_{x,y^1, \ldots,y^p} \; \begin{pmatrix}
y^1 \\ \vdots \\y^p
\end{pmatrix}   \tag{$\mathcal{VP}_p$} \\
  \\
                        & \quad \textup{s.t.} \;\; y^i  \in F(x) \textrm{ for each }  i \in [p], \\
                        & \qquad \quad x \in \Omega
\end{array} 
\end{equation} is considered. Some of the main results in that reference that will be employed by us are summarized in the theorem below. There,  for $\varepsilon \geq  0,$ the projection of the set $\varepsilon\textrm{-}\wargmin \eqref{vpp}$ onto $\R^n$ is  denoted by $\varepsilon\textrm{-}\wargmin_x\eqref{vpp}$,  i.e.,
$$\varepsilon\textrm{-}\wargmin_x\eqref{vpp}:= \{x' \in \R^n \mid \exists\; y^1,\ldots y^p \in F(x') : (x',y^1,\ldots, y^p) \in \varepsilon\textrm{-}\wargmin \eqref{vpp}\}.$$
We recall here Prop.\ 3.1(i), Theorem 3.5(i), Cor.\ 3.8(i), Theorem 4.4 and 4.5 from \cite{EichfelderQuintanaRocktaeschel2022}:
\begin{Theorem}
\label{thm:splmainresults}
Suppose that the set-valued mapping $F$ is compact-valued. Then, the following statements hold:
\begin{enumerate}
\item \label{item:monotwargminvpp}For every $\varepsilon \geq 0,$ the sequence $\left\{\varepsilon \textrm{-} \wargmin_x \eqref{vpp}\right\}_{p\geq 1}$ is monotonically increasing in $p$ with respect to inclusion, i.e., 
$$p_1 \leq p_2 \Longrightarrow \varepsilon \textrm{-} \wargmin_x \;(\mathcal{VP}_{p_1}) \subseteq \varepsilon \textrm{-} \wargmin_x \;(\mathcal{VP}_{p_2}).$$ 
\item \label{item:approxpropertyvpp}The following approximation property is true:
\begin{displaymath}
\bigcup\limits_{p \in \N} \wargmin_x \eqref{vpp} \subseteq \wargmin^l \eqref{sp} = \bigcap\limits_{\varepsilon >0}\bigcup\limits_{p\in \N} \varepsilon \textrm{-}\wargmin_x  \eqref{vpp}.
\end{displaymath} Furthermore, if $F \left(\wargmin^l \eqref{sp}\right)$ is bounded, then for each $\varepsilon>0$ we can find $p \in \N$ such that
\begin{displaymath}
\wargmin_x \eqref{vpp} \subseteq \wargmin^l \eqref{sp} \subseteq \varepsilon \textrm{-} \wargmin_x \eqref{vpp}.
\end{displaymath}
\item \label{item:fdvpvpp}Suppose that $|\Omega|< + \infty$ or that $\max\limits_{x\in \Omega} |\Min(F(x),\R^m_+)|< + \infty.$ Then, there exists $p \in \N$ such that
\begin{equation}\label{eq:wfdvpspl}
    \wargmin_x \eqref{vpp} = \wargmin^l \eqref{sp}.
\end{equation} Furthermore, in the first case \eqref{eq:wfdvpspl} holds with $p = |\Omega| - 1;$ in the second case with $p = \max\limits_{x\in \Omega} |\Min(F(x),\R^m_+)|.$
\end{enumerate}
\end{Theorem}

The existence of $p \in \N$ such that \eqref{eq:wfdvpspl} holds is known as the weakly minimal finite dimensional vectorization property, see \cite{EichfelderQuintanaRocktaeschel2022}. A version of this property will be introduced in our context later on. The following theorem from  \cite[Theorem 3.9]{EichfelderQuintanaRocktaeschel2022} is another result on the approximation quality of $ \wargmin_x  \eqref{vpp}$ which we need in the following. 

\begin{Theorem}\label{thm:sufficient vectorization}
Suppose that $\Omega$ is compact, the mapping $F$ is compact-valued, the set $F\left(\Omega\right)$ is bounded, and  $\gph F$ is closed. Then,
\begin{displaymath}
\forall\; x\in \Omega, \exists\; \bar{x} \in \cl \left(\bigcup\limits_{p \in \N} \wargmin_x  \eqref{vpp}\right) : F(\bar{x})\preceq^l F(x).
\end{displaymath}
\end{Theorem}

\subsection{The Uncertain Multiobjective Problem}

In the final part of this section, we introduce the main problem of the paper. This will be done under the following setup:

\begin{Assumption}\label{basic_setup}
Let $\Omega \subseteq \R^n$ and $\mathcal{U}\subseteq \R^k$ be nonempty and closed sets, with $\mathcal{U}$ compact. Furthermore, let $f:\Omega \times \mathcal{U} \rightarrow \R^m$ be a given continuous vector-valued function.
\end{Assumption}

Under Assumption \ref{basic_setup}, we study the uncertain multiobjective optimization problem 
\begin{equation*}\label{ump}
\left\{\begin{array}{ll}
\min\limits_{x} \; f(x,u)  \tag{$\mathcal{UP}$} \\
  \; \textup{s.t.} \;\; x \in \Omega\\
\end{array}  \Big|\; u \in \mathcal{U} \right\}.
\end{equation*} Thus, \eqref{ump} is formally represented as a parametric family of multiobjective optimization problems, where there is one such problem for each possible scenario $u \in \mathcal{U}.$ The parameter $u$ in this case represents those in a model which are not available at optimization time, but for which we have as information that they belong to the uncertainty set $\mathcal{U}.$ In a more general setting, there are also uncertainties in the constraints. 

As mentioned in the introduction, the main ingredient needed in order to define solution concepts for \eqref{ump} by using the robust approach is the so called robust counterpart problem. In this paper we follow the min-max set-based approach introduced in \cite{EhrgottIdeSchobel2014} and consider the associated set-valued mapping $F_{\mathcal{U}}: \Omega \rightrightarrows \R^m$ defined as 
\begin{equation}\label{eq:F_U}
F_{\mathcal{U}}(x):= \left\{f(x,u) \mid u \in \mathcal{U}\right\}.
\end{equation} 
For every $x \in \Omega,$ the set  $F_{\mathcal{U}}(x)$ represents the set of all possible vector costs associated to the decision $x.$  The robust counterpart of \eqref{ump} is then the set optimization problem
\begin{equation}\label{rc}\tag{$\mathcal{RC}$}
\begin{array}{rl}
&\min\limits_{x} \; F_\mathcal{U}(x) \\
                        &\;\textup{s.t.} \;\; x \in \Omega.\\
\end{array}  
\end{equation} 

\begin{Definition}(\cite[Definition 3.1]{EhrgottIdeSchobel2014})\label{def:robustweaklyeff}
Let Assumption \ref{basic_setup} hold  and let $\bar{x} \in \Omega.$  We say that $\bar{x}$ is a robust weakly efficient solution of \eqref{ump} if it is a $u$-weakly efficient solution of \eqref{rc},  i.e.,
\begin{equation*}
\nexists \; x\in \Omega: F_{\mathcal{U}}(x) \subseteq F_{\mathcal{U}}(\bar{x}) - \Int \R_+^m,
\end{equation*} 
or, equivalently,
\begin{equation*}
\forall \; x\in \Omega: F_{\mathcal{U}}(x) \nprec^u F_{\mathcal{U}}(\bar{x}).  
\end{equation*} 
 The set of robust weakly efficient solutions of \eqref{ump} will be denoted by $\wargmin \eqref{ump}.$ 

\end{Definition}
Note that, by definition, we have $\wargmin \eqref{ump} = \wargmin^u \eqref{rc}.$ It is  worth mentioning that further concepts of robust efficiency and robust strict efficiency exist, but they are out of the scope of this work. For a discussion on different concepts  see, for instance, \cite{Ide2016,Schoebel21}. 

We concentrate here only on robust weakly efficient points as they have richer theoretical properties, and we hope that a more detailed analysis for the other concepts can be later pursued.
\begin{Remark}
In the single-objective case, i.e., for $m = 1,$ it can be shown that $\bar{x} \in \wargmin \eqref{ump}$ if and only if $\bar{x}$ is an optimal solution of the problem 
        \begin{equation} \label{rc scalar}
        \begin{array}{rl}
        &\min\limits_{x} \; \sup\limits_{u \in \mathcal{U}}f(x,u) \\
                            &\;\textup{s.t.} \;\; x \in \Omega.\\
        \end{array}
\end{equation} Problem \eqref{rc scalar} is well known as the classical robust counterpart problem of \eqref{ump} in this case, see \cite{bentalghaouinemirovski2009}. Thus, \eqref{rc} is truly an extension of the min-max approach to the multiobjective setting. 
\end{Remark}

\section{A Connection between Set Relations}\label{section:vectorization}

The results in  \cite{EichfelderQuintanaRocktaeschel2022} are based on comparing sets with the lower type set relation and they cannot  be extended to other binary relations for sets in a  straightforward way. For robust weak efficiency, see Definition \ref{def:robustweaklyeff}, we work with the upper type set relation. 
For that reason we provide in this section results which give equivalent characterizations of $A\prec^u B$ by using the set relation $\prec^l$. 

In  \cite[Proposition 3.1 $(iv)$]{Baotammer2019} it was shown that  for any nonempty sets $A,B \subseteq \R^m$ the equivalence $$A \preceq^u B \Longleftrightarrow \left(A - \R^m_+\right)^c \preceq^l \left(B - \R^m_+\right)^c $$ holds.

However, we need an equivalent reformulation which is still based on comparing compact sets as this is required for applying Theorem \ref{thm:splmainresults}. For that reason we need more complicated constructions in this section by introducing suitable intersections. 
For the proofs we make use of the following functional.

To any nonempty compact set $A \subseteq \R^m,$ we define the functional  $\psi_A: \R^m \rightarrow \R$ by
\begin{equation}\label{eq:psiB}  
\psi_A(y):= \min\limits_{a \in A} \max_{j \in [m]} \{y_j-a_j\},
\end{equation} 
which is related to minimizing a special case of the  well-known Tammer-Weidner functional \cite{Tammerfuntcional}. 
We need this functional and its properties for proving our results on the binary relations for sets. 
Note that, in particular, we have $\psi_{\{0\}}(y) = \max_{j \in [m]} \{y_j\}.$ For simplicity, we will write $\psi$ in place of  $\psi_{\{0\}}.$ Furthermore, it is clear that 
\begin{equation}\label{eq:representabilitymax}
    -\R^m_+= \{y \in \R^m \mid \psi(y)\leq 0 \}, \quad - \Int \R^m_+ = \{y \in \R^m \mid \psi(y)< 0 \}.
\end{equation}

\begin{Lemma}\label{lem:psiB}
Let $A \subseteq \R^m$ be nonempty and compact, and consider the associated functional $\psi_A$ given in \eqref{eq:psiB}. Then, the following statements hold:
\begin{enumerate}
\item \label{item:propertiescontpsiB} $\psi_A$ is Lipschitz continuous on $\R^m,$ i.e., there exists a constant  $L > 0$ such that
$$\forall \; y^1, y^2 \in \R^m: |\psi_A(y^1) - \psi_A(y^2)|\leq L\|y^1 - y^2\|,$$
\item \label{item:propertiesmonotonpsiB} $\psi_A$ is $\R^m_+$-monotone, i.e., 
$\forall\; y^1, y^2 \in \R^m: y^1 \leq y^2 \Longrightarrow \psi_A(y^1) \leq \psi_A(y^2),$
\item \label{item:propertiestranslationinvpsiB} $\psi_A$  satisfies the translation invariance property with respect to $e,$ i.e., 
$$\forall\; y \in \R^m, t \in \R: \psi_A(y + te) = \psi_A(y) + t,$$
\item \label{item:representationB-K} $A- \R^m_+ = \left\{y \in \R^m \mid \psi_A(y)\leq 0\right\},$
\item \label{item:representationclB-Kc} $\cl \left(A- \R^m_+\right)^c = \left\{ y \in \R^m \mid \psi_A(y) \geq 0\right\},$
\end{enumerate}
\end{Lemma}
\begin{proof} \ref{item:propertiescontpsiB}, \ref{item:propertiesmonotonpsiB}, \ref{item:propertiestranslationinvpsiB} If $A = \{0\},$ it is easy to see that $\psi$ satisfies the properties. For the general case we can take into account that $A$ is compact to deduce these properties from those of the functional  $\psi$.

\ref{item:representationB-K} Let $y \in A- \R^m_+.$ Then, $y = a- k$ for some $a \in A$ and $k \in \R^m_+.$ Taking into account \eqref{eq:representabilitymax}, we now deduce that  
$$\psi_A(y) \leq \psi(y-a) = \psi(-k) \leq 0.$$  Thus, $A- \R^m_+ \subseteq \left\{y \in \R^m \mid \psi_A(y)\leq 0\right\}.$ 

Suppose now that $y \in \R^m$ is given with $\psi_A(y) \leq 0.$ Then, because $A$ is compact and $\psi$ is continuous, there exists $a \in A$ such that $\psi_A(y) = \psi(y-a).$ Therefore, we get  $\psi(y-a) \leq 0,$ which implies $y\in \{a\} -\R^m_+ \subseteq A-\R^m_+$ according to \eqref{eq:representabilitymax}.

\ref{item:representationclB-Kc} From statement \ref{item:representationB-K}, we have $\cl \left(A- \R^m_+\right)^c = \cl \left\{y \in \R^m \mid \psi_A(y)> 0\right\}.$ Therefore, in order to finish the proof, we just need to show that  $$\cl \left\{y \in \R^m \mid \psi_A(y)> 0\right\} =  \left\{ y \in \R^m \mid \psi_A(y) \geq 0\right\}.$$ The inclusion $\cl \left\{y \in \R^m \mid \psi_A(y)> 0\right\} \subseteq  \left\{ y \in \R^m \mid \psi_A(y) \geq 0\right\}$ follows from the continuity of $\psi_A$ in statement \ref{item:propertiescontpsiB}. In order to see the second inclusion, fix $\bar y \in \R^m$ such that $\psi_A(\bar y)\geq 0.$ If $\psi_A(\bar y) > 0$ there is nothing to prove. %, as we automatically have $y \in \cl \left\{y \in \R^m \mid \psi_A(y)> 0\right\}.$ 
Suppose now that $\psi_A(\bar y) =0.$ Then, according to statement \ref{item:propertiestranslationinvpsiB}, for every $t >0$ we have $$\psi_A(\bar y + te) = \psi_A(\bar y)+ t = t >0.$$ Thus, $\bar y+te \in  \left\{y \in \R^m \mid \psi_A(y)> 0\right\}$ for any $t>0.$ By letting $t \to 0,$ we deduce that $\bar y \in \cl \left\{y \in \R^m \mid \psi_A(y)> 0\right\},$ as desired.
\end{proof}

By using sublevel sets of this functional we can give a first reformulation for the set relation $\prec^u$.

\begin{Lemma}\label{lem:AulesB}
Let $A,B \subseteq \R^m$ be nonempty compact sets. Then, 

\begin{equation*}
A \prec^u B \Longleftrightarrow \exists\; \varepsilon >0 : A- \R^m_+ \subseteq \left\{y \in \R^m \mid \psi_B(y) \leq - \varepsilon\right\}.
\end{equation*}
\end{Lemma}
\begin{proof}
Suppose first that $A \prec^u B$ and set $\varepsilon:= - \sup_{a\in A} \psi_B(a).$ We claim that $\varepsilon >0.$ Indeed, because $A$ is compact and $\psi_B$ is continuous according to Lemma \ref{lem:psiB} \ref{item:propertiescontpsiB}, there exists an element  $\bar{a} \in A$ such that 

\begin{equation}\label{eq:infattainedpsiBwrtA}
\psi_B(\bar{a}) = \sup\limits_{a\in A} \psi_B(a).
\end{equation} Now, since $A \prec^u B,$ we can find an element $\bar{b} \in B$ such that $\bar{a} < \bar{b}.$ It then follows from \eqref{eq:infattainedpsiBwrtA} and \eqref{eq:representabilitymax} that
$$\varepsilon =  - \psi_B(\bar{a}) =  - \min\limits_{b\in B} \psi(\bar{a}-b) \geq -\psi(\bar{a}-\bar{b}) > 0.$$ Next, by the definition of $\varepsilon,$ we get that $A \subseteq \left\{y \in \R^m \mid \psi_B(y) \leq - \varepsilon\right\}.$ Using the $\R^m_+$- monotonicity of $\psi_B$ from Lemma \ref{lem:psiB}, we then deduce that 
\begin{equation}\label{eq:inclusiona-KepsilonpsiB}
A- \R^m_+ \subseteq \left\{y \in \R^m \mid \psi_B(y) \leq - \varepsilon\right\}.
\end{equation}

Conversely, assume that for some $\varepsilon >0$ the inclusion \eqref{eq:inclusiona-KepsilonpsiB} holds and fix $a \in A.$ Then, from Lemma \ref{lem:psiB} \ref{item:propertiestranslationinvpsiB} and \eqref{eq:inclusiona-KepsilonpsiB} we find that 

\begin{equation}\label{eq:a+epsiloneinB-K}
\psi_B(a+ \varepsilon e) = \psi_B(a) +\varepsilon \leq - \varepsilon +\varepsilon = 0.
\end{equation} According to Lemma \ref{lem:psiB} \ref{item:representationB-K}, the inequality in \eqref{eq:a+epsiloneinB-K} now implies that $a+ \varepsilon e \in B- \R^m_+.$ Thus, we obtain that
\begin{equation}\label{eq:ainB-IntK}
a = a + \varepsilon e - \varepsilon e \subseteq  B-\R^m_+ - \{\varepsilon e \}\subseteq  B -\Int \R^m_+.
\end{equation} Since $a$ was arbitrarily chosen in $A,$ the inclusion \eqref{eq:ainB-IntK}  proves that $A \prec^u B.$
\end{proof}

In the rest of the paper, for a  closed, convex, pointed  and solid cone $C \subseteq \R^m$ and arbitrary elements $y^1,y^2 \in \R^m,$ we define  the sets 
$$\left[y^1, y^2\right]_C:= \left(\{y^1\} + C\right) \cap \left(\{y^2 \}-  \R^m_+\right),
$$
and $$
 \left(y^1, y^2\right)_C:= \left(\{y^1\} + \Int C\right) \cap \left(\{ y^2 \}-  \Int \R^m_+\right).$$ 
Recall that a set $C\subseteq\R^m$ is a cone if $\lambda\geq 0$ and $c\in C$ imply $\lambda c\in C$. It is called pointed if $C\cap (-C)=\{0\}$ and solid if $\Int C\neq\emptyset$. For the cone $C$ we will make use of a set which is strictly contained in the ordering cone $\R^m_+$ in the sense that $C\setminus\{0\}\subseteq \Int \R^m_+$. As a consequence, the above sets satisfy 
 $\left[y^1, y^2\right]_C\subseteq  \left[y^1, y^2\right]:=\{y\in\R^m\mid y^1\leq y\leq y^2\}$ and 
$\left(y^1, y^2\right)_C\subseteq  \{y\in\R^m\mid y^1< y<  y^2\}$. 

\begin{Lemma}\label{lem:shiftinglemma}
Let $A \subseteq \R^m$ be nonempty and compact, and $C \subseteq \left(\Int \R^m_+\right)\cup \{0\} $ be a closed, convex, pointed and solid cone. Suppose that the elements $\ell b, ub(A) \in \R^m$ are such that $A \subseteq \left(\ell b,ub(A)\right)_C.$ Then, 
$$\cl \left(A- \R^m_+\right)^c\cap \left(\{\ell b\} + C\right) \subseteq \left( \cl \left(A- \R^m_+\right)^c\cap \left[\ell b, ub(A)\right]_C\right)  + \R^m_+.$$
\end{Lemma}

\begin{proof}
Let $y \in \cl \left(A- \R^m_+\right)^c\cap \left(\{\ell b\} + C \right).$  Then, we have  $v:= y- \ell b \in C$ and,
according to Lemma \ref{lem:psiB} \ref{item:representationclB-Kc},
that 
$\psi_A(y) \geq 0$. Furthermore, since $A \subseteq \{\ell b\} + \Int C$ and $A$ is nonempty, we can take an element $a \in A$ to deduce that
\begin{equation}\label{eq:psiy^Ly<0}
\psi_A(y- v) = \psi_A(\ell b) \leq \psi(\ell b-a) < 0,
\end{equation} where the last inequality follows from \eqref{eq:representabilitymax} and the fact that $\Int C \subseteq \Int \R^m_+.$ Applying now the intermediate value theorem, we get the existence of $\bar{t} \in [0,1)$ such that $\bar{y}:= y- \bar{t}v$ satisfies $\psi_A(\bar{y}) = 0.$ Thus, from Lemma \ref{lem:psiB} \ref{item:representationclB-Kc}, we obtain that $\bar{y} \in  \cl \left(A- \R^m_+\right)^c.$ Applying also Lemma \ref{lem:psiB} \ref{item:representationB-K}, we get that $\bar{y} \in A-\R^m_+ \subseteq \{ub(A)\}- \R^m_+.$ Furthermore, because $\bar{t}<1,$ we find that
$$\bar{y} - \ell b =y-\bar t v-\ell b= y- \bar{t}(y-\ell b)-\ell b = (1-\bar t)(y-\ell b)=(1-\bar{t})v \in C.$$ This shows that $\bar{y} \in \cl \left(A- \R^m_+\right)^c\cap \left[\ell b,ub(A)\right]_C .$ The statement then follows from the fact that $y-\bar{y} = \bar{t}v \in C$ and $C\subseteq \R^m_+$.
\end{proof}

\begin{Proposition}\label{prop: cl A-K equivalent}
Let $A \subseteq \R^m$ be nonempty and compact. Then,
\begin{equation*}
    \cl \left(A- \R^m_+\right)^c =  \left(A- \Int \R^m_+\right)^c.
\end{equation*}
\end{Proposition}
\begin{proof}
First, from $A - \Int \R^m_+  \subseteq  A - \R^m_+$ we deduce that $\left(A- \R^m_+\right)^c \subseteq \left(A- \Int \R^m_+\right)^c.$ Therefore, by  taking the closure on both sides of this inclusion we  obtain 
$$\cl \left(A- \R^m_+\right)^c \subseteq \left(A- \Int \R^m_+\right)^c.$$ 

In order to see the other inclusion, take any $y \in \left(A- \Int \R^m_+\right)^c$ and consider the sequence $\{y + k^{-1}e\}_{k \in \N}.$ Then, $\{y + k^{-1}e\}_{k \in \N} \subseteq \left(A- \R^m_+\right)^c.$ Otherwise, for some $k \in \N$  we have $y + k^{-1}e \in A- \R^m_+,$ which would imply $y \in A - \R^m_+ - \{k^{-1}e\} \subseteq A - \Int \R^m_+,$ a contradiction. Since $y = \lim_{k \to \infty} y + k^{-1}e$ and $y$ was chosen arbitrarily, it follows that $\left(A- \Int \R^m_+\right)^c \subseteq  \cl \left(A- \R^m_+\right)^c.$
\end{proof}

The next theorem gives the desired reformulation of the binary relation $\prec^u$ by using $\prec ^l$, for compact sets $A$ and $B$ for which a strict common lower bound $\ell b\in\R^m$ w.r.t.\ $C$ exists in the sense that $A\subseteq \{\ell b\}+\Int C$ and $B\subseteq \{\ell b\}+\Int C$. 
The existence of upper bounds $ub(A)$ and $ub(B)$ with 
$A\subseteq \{ub(A)\}-\Int \R^m_+$ and $B\subseteq \{ub(B)\}-\Int \R^m_+$ is not restrictive, as the sets $A$ and $B$ have been assumed to be compact and one can take for $ub(A)$, for instance, a point $u\in\R^m$ which is component-wise defined by 
$u_j:=(\max_{a\in A}a_j)+\gamma$ for $j\in[m]$ and some  $\gamma>0$. For an illustration of the equivalence of the strict binary relations of sets from the next theorem see Figure \ref{fig:equiv}.

\begin{figure}[h]
	\centering
	\begin{tikzpicture}[scale=1.4]%.7]
	
	% Cones and Shades
	%\shade[lower left = lightgray, upper right = white] (-0.5,3.5) -- (-1.25,-1.25) -- (4.5,-0.5) -- (4.5, 3.5);
	%\shade[lower left = lightgray, upper right = white] (-0.9342105263157894,3.5) -- (-0.9342105263157894, 0.75)  -- (0, 0.75) -- (0.75, 0) -- (0.75, -0.9891304347826086) -- (4.5, -0.9891304347826086) -- (4.5, 3.5);
	
	\draw[thick, fill = gray!100]  (1.25,1.25) -- (-0.8552631578947368, 1.25) -- (-0.9342105263157894, 0.75) -- (0, 0.75) -- (0.75, 0) --(0.75, -0.9891304347826086) -- (1.25, -0.9239130434782608) -- (1.25, 1.25);
	\draw[thick, fill = gray!100]  (2.3, 1.25) -- (-0.8552631578947368, 1.25) -- (-0.9342105263157894, 0.75) -- (0, 0.75) -- (0.75, 0) --(0.75, -0.9891304347826086) -- (2.3, -0.79) -- (2.3, 1.25);
	%\path[thick, pattern=dots, pattern color=black]  (1.25,1.25) -- (-0.8552631578947368, 1.25) -- (-0.9342105263157894, 0.75) -- (0, 0.75) -- (0.75, 0) --(0.75, -0.9891304347826086) -- (1.25, -0.9239130434782608) -- (1.25, 1.25);
	
	\draw[thick, fill = gray!100] (4.25,3.25) -- (-0.5394736842105262, 3.25) -- (-0.6578947368421052, 2.5) -- (3, 2.5) -- (3.5, 2) -- (3.5, -0.6304347826086956) -- (4.25, -0.5326086956521738) -- (4.25,3.25);
	%\path[thick, pattern=dots, pattern color=black] (4.25,3.25) -- (-0.5394736842105262, 3.25) -- (-0.6578947368421052, 2.5) -- (3, 2.5) -- (3.5, 2) -- (3.5, -0.6304347826086956) -- (4.25, -0.5326086956521738) -- (4.25,3.25) ;

    % Basic Sets
	%\shade[lower left=gray, upper right = white] (-1,2) -- (-1,0) -- (-0.7,-0.7) -- (0,-1) -- (2,-1) -- (2,2);
	\draw[thick, fill = gray!50] (0,0) circle(0.75);
	\draw[thick, fill = gray!50] (3,2) circle(0.5);
	\node[black] at (0,0){$A$};
	\node[black] at (3,2){$B$};
	
	% Points
	\filldraw[thick] (-1.25,-1.25) circle(1pt);
	\node[below left] at (-1.25,-1.25){$\ell b$};
	
	\filldraw[thick] (2.3,1.25) circle(1pt);
	%\node[above] at (1.25, 1.25){$ub(A)$};
	\node[above] at (2.3, 1.25){$ub(A)$};
	
	\filldraw[thick] (4.25,3.25) circle(1pt);
	\node[above] at (4.25,3.25){$ub(B)$};
	
	% Dashed lines
	\draw[dashed] (3,2.5) -- (-1.5,2.5);
	\draw[dashed] (3.5,2) -- (3.5,-1.5);
	
	\draw[thick] (-0.5,3.5) -- (-1.25,-1.25); 
	\draw[thick] (-1.25,-1.25) -- (4.5,-0.5);
	
	\draw[dashed] (-0.93,0.75) -- (-0.93,3.5);
	\draw[dashed] (0.75, -0.99) -- (4.5, -0.99);
	
	% Captions
	\node[above right] at (-1.2,-1.1){$\ell b + C$};
	\node[] at (0.75,1){$\cl \left(A- \R^m_+\right)^c \cap \left[\ell b,ub(A)\right]_C$};
	\node[] at (1.8,2.9){$\cl \left(B- \R^m_+\right)^c \cap \left[\ell b,ub(B)\right]_C$};
	
	\end{tikzpicture}
	
	\caption{Illustration of the equivalence in Theorem \ref{thm:uppertolower}\label{fig:equiv}}
\end{figure}
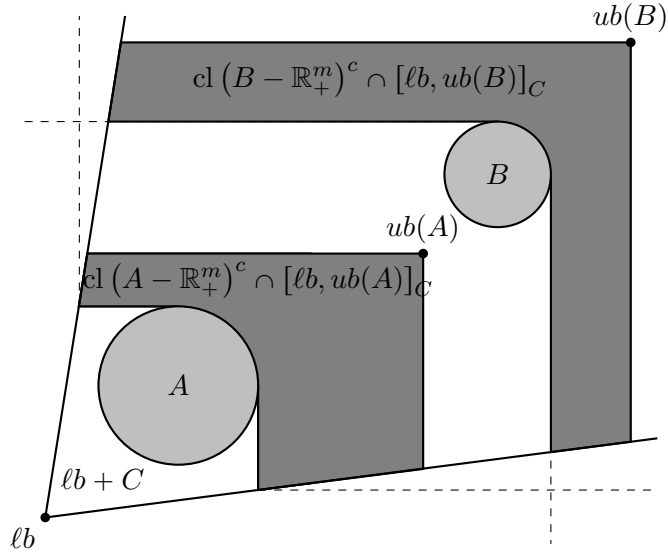

\begin{Theorem}\label{thm:uppertolower}
Let $A,B \subseteq \R^m$ be nonempty compact sets. Suppose that there exists  a closed, convex, pointed and solid cone $C \subseteq (\Int \R^m_+) \cup \{0\},$ together with elements $\ell b, ub(A),ub(B) \in \R^m,$  such that $ A \subseteq \left(\ell b,ub(A)\right)_C$ and $B \subseteq \left(\ell b,ub(B)\right)_C. $  Then, the following  statements are equivalent:
\begin{enumerate}
\item \label{item:AulessB} $A\prec^u B,$
%\item \label{item:AulessBcomplement} $\cl \left(A- \R^m_+\right)^c  \prec^l \cl \left(B- \R^m_+\right)^c,$
\item \label{item:AulessBcomplementintersection}$\cl \left(A- \R^m_+\right)^c \cap \left[\ell b,ub(A)\right]_C \prec^l \cl \left(B- \R^m_+\right)^c\cap \left[\ell b,ub(B)\right]_C.$ 
\end{enumerate}
\end{Theorem}

\begin{proof} Note that from Lemma \ref{lem:psiB} \ref{item:representationclB-Kc} we have 
\begin{equation}\label{eq:complementclosureBfinished}
\cl \left(B- \R^m_+\right)^c \cap \left[\ell b,ub(B)\right]_C = \left\{y \in \R^m \mid \psi_B(y)\geq 0\right\} \cap \left[\ell b,ub(B)\right]_C.
\end{equation}

We start by showing that \ref{item:AulessB} implies \ref{item:AulessBcomplementintersection} and hence we suppose 
that $A \prec^u B.$ Then, according to Lemma \ref{lem:AulesB}, we can find $\varepsilon > 0$ such that $A- \R^m_+ \subseteq \left\{y \in \R^m \mid \psi_B(y) \leq - \varepsilon\right\}.$ Taking complements and closure in this inclusion,  we get
\begin{equation}\label{eq:complementclosureA}
\cl \left\{y \in \R^m \mid \psi_B(y) > - \varepsilon\right\} \subseteq   \cl \left(A- \R^m_+\right)^c.
\end{equation} Furthermore, arguing similarly to the proof of Lemma \ref{lem:psiB} \ref{item:representationclB-Kc}, we find that 
$$\cl \left\{y \in \R^m \mid \psi_B(y) > - \varepsilon\right\} = \left\{y \in \R^m \mid \psi_B(y) \geq - \varepsilon\right\}.$$  Thus, \eqref{eq:complementclosureA} is equivalent to the inclusion 
\begin{equation}\label{eq:complementpsiBgeqepsilon}
\left\{y \in \R^m \mid \psi_B(y) \geq - \varepsilon\right\} \subseteq   \cl \left(A- \R^m_+\right)^c,
\end{equation} which implies 
\begin{equation}\label{eq:complementclosureAfinished}
 \left\{y \in \R^m \mid \psi_B(y)\geq -\varepsilon\right\} \cap \left[\ell b,ub(A)\right]_C \subseteq  \cl \left(A- \R^m_+\right)^c \cap \left[\ell b,ub(A)\right]_C.
\end{equation} Fix now $y \in  \cl \left(B- \R^m_+\right)^c\cap \left[\ell b,ub(B)\right]_C$ and the element $v:= y - \ell b \in C.$ Then, because $B \subseteq \{\ell b\} + \Int C,$ it is clear that $\ell b \in  B- \Int  \R^m_+.$  As $y\in \cl \left(B- \R^m_+\right)^c$, we have by Proposition \ref{prop: cl A-K equivalent} that  $v\neq 0,$ and, as $v\in C\subseteq (\Int \R^m_+) \cup \{0\}$, that $v \in \Int \R^m_+.$ 

Now, according to \eqref{eq:complementclosureBfinished}, we know that $\psi_B(y)\geq 0.$ Therefore, using the continuity of $\psi_B$ by Lemma \ref{lem:psiB} \ref{item:propertiescontpsiB}, we deduce the existence  of $\bar{t} \in (0,1)$ such that the vector $\bar{y}':= y- \bar{t}v$ satisfies $\psi_B(\bar{y}')\geq -\varepsilon.$ On the other hand, for such $\bar{t},$ we also get
\begin{equation}\label{eq:differences}
\bar{y}' - \ell b = y- \bar{t}(y-\ell b)-\ell b = (1-\bar{t})v \in C.
\end{equation} Thus, the element $\bar{y}'$ satisfies $\psi_B(\bar{y}')\geq -\varepsilon$ and $  \bar{y}' \in \{\ell b\} + C.$ Taking into account \eqref{eq:complementpsiBgeqepsilon}, we then obtain
\begin{equation}\label{eq:y^tH(A)c}
\bar{y}' \in \cl \left(A- \R^m_+\right)^c \cap \left(\{\ell b \}+ C\right).
\end{equation}  We can now  apply Lemma \ref{lem:shiftinglemma} to find an element $\bar{y} \in \cl \left(A- \R^m_+\right)^c\cap \left[\ell b,ub(A)\right]_C $ such that $\bar{y} \leq \bar{y}'.$ Since $v  \in \Int \R^m_+,$ this then implies 
$\bar{y}\leq \bar{y}' < y,$ as desired. As $y$ was chosen arbitrarily we are done with showing that  \ref{item:AulessB} implies \ref{item:AulessBcomplementintersection}.

It remains to show that  \ref{item:AulessBcomplementintersection} implies   \ref{item:AulessB}. For doing so, we prove that  $A \nprec^u B$  implies that \ref{item:AulessBcomplementintersection} is not satisfied. If $A \nprec^u B$, it is possible to find $a \in A$ such that $a \notin B - \Int \R^m_+.$ Thus, by Proposition \ref{prop: cl A-K equivalent}, we have $a \in \cl \left(B- \R^m_+\right)^c\cap \left(\{\ell b\}+C\right).$ Applying now Lemma \ref{lem:shiftinglemma}, we find $\bar{a} \in   \cl \left(B- \R^m_+\right)^c\cap \left[\ell b,ub(B)\right]_C$ such that $\bar{a} \leq a.$ In order to finish the proof, it suffices to show that
$$\bar{a} \notin\left(  \cl \left(A- \R^m_+\right)^c\cap \left[\ell b,ub(A)\right]_C \right)+\Int \R^m_+.$$ This is indeed the case, as otherwise there exists $\bar{y} \in  \cl \left(A- \R^m_+\right)^c\cap \left[\ell b,ub(A)\right]_C $ such that $\bar{y} < \bar{a}.$ Therefore, we can find a sequence $\{y^k\}_{k \in \N} \subseteq \left(A- \R^m_+\right)^c$ such that $y^k \to \bar{y}.$ Because $\bar{y} < \bar{a},$ we then have that $y^k < \bar{a}$ for $k$ large enough. We then deduce that
$$y^k \in \{\bar{a}\} - \Int \R^m_+ \subseteq\{ a\} - \R^m_+ -\Int \R^m_+  \subseteq A- \R^m_+,$$
a contradiction. The proof is complete.
\end{proof}

The following example shows that the use of a cone $C\subseteq (\Int \R^m_+)\cup \{0\}$ is essential for the result in Theorem \ref{thm:uppertolower}.

\begin{Example}
For $A=\{y\in\R^2\mid \|y-(-1,0)^\top\|_2\leq 1\}$ and
$B=\{y\in\R^2\mid \|y-(2,0)^\top\|_2\leq 2\}$
it holds
$A\prec^uB$. For $C=\R^m_+$ and $\ell b=(-2.1, -2.1)^\top$, $ub(A)=(0.1,1.1)^\top$, and $ub(B)=(4.1,2.1)^\top$ we have that $A\subseteq (\ell b, ub(A))_C$ and $B\subseteq (\ell b, ub(B))_C$, but only 
$\cl \left(A- \R^m_+\right)^c \cap \left[\ell b,ub(A)\right]_C \preceq^l \cl \left(B- \R^m_+\right)^c\cap \left[\ell b,ub(B)\right]_C$ and not $$\cl \left(A- \R^m_+\right)^c \cap \left[\ell b,ub(A)\right]_C \prec^l \cl \left(B- \R^m_+\right)^c\cap \left[\ell b,ub(B)\right]_C.$$ 
\end{Example}

\section{Epigraphical Reformulations}

This section is devoted to the  formalization of an epigraphical reformulation for \eqref{ump} and the description of its main properties. As a motivation we start by recalling that  for uncertain single-objective optimization problems we have that 
\begin{equation}\label{eq:epi ref scalar}
\begin{array}{ll}
\min\limits_{x,y} \; y   \\
\;\;  \textup{s.t.} \;\; f(x,u)\leq y  \; \;\;\forall \; u\in \mathcal{U}, \\
\;\; \;\;\;\;\;\;\; x \in \Omega
\end{array} 
\end{equation} is an equivalent formulation of the already discussed (single-objective) robust counterpart  \eqref{rc scalar}. This is meant  in the sense that some point $\bar{x}$ solves \eqref{rc scalar} if and only if there exists some  $\bar{y} \in \R$ such that $ (\bar{x}, \bar{y})$   solves  \eqref{eq:epi ref scalar}. Considering \eqref{eq:epi ref scalar} instead of \eqref{rc scalar} has the important  theoretical advantage that semi-infinite methods can be applied in general nonconvex settings. Moreover,  in case of  convexity, reformulations of the semi-infinite constraints are possible and already well studied, see for example \cite{BenTalHertogVial2015, bentalghaouinemirovski2009}.

Thus, such an epigraphical reformulation is very attractive from the computational point of view. The issue is that \eqref{rc} is now a set optimization problem and it is not clear what the epigraphical reformulation is in this more general case. In order to extend \eqref{eq:epi ref scalar} to deal with the more general robust counterpart \eqref{rc} in the multiobjective setting,  the following concept of  proper boundedness will be important for being able to apply the results from Theorem \ref{thm:uppertolower}.

\begin{Definition}\label{def:properlowerboundedness}
Let Assumption \ref{basic_setup} hold. We say that the function $f$ is properly lower bounded on $\Omega \times \mathcal{U}$ if there exists $\ell b \in \R^m$ and a closed, convex, pointed and solid cone $C$ such that $C \subseteq (\Int \R^m_+) \cup \{0\} $ and $f \left(\Omega \times \mathcal{U}\right) \subseteq \{\ell b \}+ \Int C.$ In that case, we also say that $f$ is properly lower bounded on $\Omega \times \mathcal{U}$ with respect to $\ell b$ and $C.$
\end{Definition}

\begin{Remark}
There are already boundedness concepts for functions which are related to the above definition of proper boundedness.
The function  $f$ is called lower bounded  on $\Omega\times \mathcal{U}$ if there exists a point $\ell b \in \R^m$ such that $f \left(\Omega\times \mathcal{U}\right) \subseteq \{\ell b\} + \Int \R^m_+.$  Moreover, the function $f$ is said to be bounded  on $\Omega\times \mathcal{U}$ if the set $f \left(\Omega\times \mathcal{U}\right)$ is a bounded subset of $\R^m.$

It is then easy to verify that it holds under our assumptions that:
\begin{enumerate}
\item \label{item:plb<lb} If the function $f$ is properly lower bounded on $\Omega\times \mathcal{U},$ then it is also lower bounded.
\item \label{item:b<plb} If the function $f$ is bounded on $\Omega\times \mathcal{U},$ then it is properly lower bounded. In particular, if $\Omega$ is compact, then $f$ is properly lower bounded on $\Omega \times \mathcal{U}$.
\item \label{item:m=1=>plb=lb} If $m = 1,$ then the concepts of lower boundedness and proper lower boundedness of $f$  on $\Omega\times \mathcal{U}$ coincide.
\end{enumerate}
\end{Remark}

In this section we work under the following assumption.
\begin{Assumption}\label{main_assumption}
In addition to Assumption \ref{basic_setup}, let $\ell b \in \R^m$ and a closed, convex, pointed, solid cone $C \subseteq \Int \R_{+}^m \cup \{0\}$ be given, and suppose that $f$ is properly lower bounded on $\Omega \times \mathcal{U}$ with respect to $\ell b$ and $C.$ 
\end{Assumption}

Next we consider  the vector-valued map $ub:\Omega \rightarrow \R^m$ defined by
\begin{equation}\label{eq:ub def}
  ub(x):= \begin{pmatrix} \sup\limits_{u \in \mathcal{U}} f_1(x,u) \\ \vdots \\ \sup\limits_{u \in \mathcal{U}} f_m(x,u) \end{pmatrix} + e
\end{equation} and the set-valued mapping $F: \Omega \rightrightarrows \R^m$ defined by 
\begin{equation}\label{eq:Fhood}
F(x):= \cl \left(F_\mathcal{U}(x)- \R^m_+\right)^c \cap \left[\ell b, ub(x)\right]_C.
\end{equation}

\begin{Remark}
    It is worth mentioning that the map  $ub(\cdot)$ already plays an important role in the context of robust multiobjective optimization when choosing the point-based minmax approach. There, more concretely, the map $\widetilde{ ub}(\cdot):=ub(\cdot)-e$ completely defines the associated robust counterpart problem, see \cite{KuroiwaLee2012}. In the setup of this paper however, $ub(\cdot)$ is used only as an upper bound map for constructing an appropriate compact-valued set-valued mapping.
\end{Remark}

Note that, due to the continuity of $f$ and the compactness of $\mathcal{U},$ the map $ub$ is well defined. Consequently, the mapping $F$ in \eqref{eq:Fhood} is  well defined. Moreover, from the proper boundedness of $f$, the definition of $F_\mathcal{U}$ in \eqref{eq:F_U} and of $F$ in \eqref{eq:ub def}, it   follows that
\begin{equation}\label{eq:F_UsubsetyU}
\forall \; x\in \Omega: F_{\mathcal{U}}(x) \subseteq (\ell b, ub(x))_C.
\end{equation}
We can then introduce the following parametric family of multiobjective problems indexed by $p \in \N$ which will be our epigraphical reformulation as we discuss next.
\begin{equation}\label{mp}
\begin{array}{ll}
\min\limits_{x,y^1, \ldots,y^p} \; \begin{pmatrix}
y^1 \\ \vdots \\y^p
\end{pmatrix}   \tag{$\mathcal{MP}_p$} \\
  \\
\quad\textup{s.t.} \;\; \min\limits_{j \in [m]}\left\{f_j(x,u) - y^i_j\right\} \leq 0 \;\; \forall \;u \in \mathcal{U}, \; i \in [p],  \\
\quad \;\;\;\;\;\;\;\;\;  y^i -\ell b \in C, \; \forall\; i \in [p],\\ 
\quad \;\;\;\;\;\;\;\;\; x \in \Omega.
\end{array} 
\end{equation}

In the single-objective case, i.e., $m=1$, %It is then clear that in the scalar case
we have, for $p = 1$ and $C = \R_+$, that \eqref{mp} is equivalent to \eqref{eq:epi ref scalar}.  
Thus, the parametric family $\left\{\eqref{mp}\right\}_{p \in \N}$   extends \eqref{eq:epi ref scalar} to the multiobjective setting. However, as we will see throughout the remaining of the section, unlike for uncertain single-objective problems,  just considering \eqref{mp} for $p = 1$ may not be enough in the general case.

\subsection{Approximation Properties}

We start by collecting some basic properties of the auxiliary set-valued mapping $F$ from \eqref{eq:Fhood}. 

\begin{Proposition} \label{prop:basic properties F}
Let Assumption \ref{main_assumption} hold. Then for $F$ as defined  in \eqref{eq:Fhood} it holds   
\begin{enumerate}
\item \label{item:F_compact valued} $F(x)$ is nonempty and compact for all $x\in\Omega$, 
\item \label{item: F(Omega) bounded} if $ \tilde{\Omega} \subseteq \Omega$ is nonempty and $f(\tilde{\Omega} \times \mathcal{U})$ is bounded, then $F(\tilde{\Omega})$ is bounded,
\item \label{item:F closed graph} $\gph F$ is closed,
\item \label{item:FtoFUweak} $\forall\; x, x' \in \Omega: F_\mathcal{U}(x) \prec^u F_\mathcal{U}(x')  \Longleftrightarrow  F(x) \prec^l F(x'),$
\item \label{item:FtoFU} $\forall\; x, x' \in \Omega: F(x) \preceq^l F(x')  \Longrightarrow  F_\mathcal{U}(x) \preceq^u F_\mathcal{U}(x') .$

\end{enumerate}
\end{Proposition}
\begin{proof} Statement 
\ref{item:F_compact valued} follows immediately from the definition. 

For 
\ref{item: F(Omega) bounded}, if $f(\tilde{\Omega} \times \mathcal{U})$ is bounded, we can  find in particular a vector $\bar{z} \in \R^m$ such that $f(\tilde{\Omega} \times \mathcal{U}) \subseteq \{\bar{z} \}- \R^m_+.$ By taking into account the definition of $F$ and the definition of $ub(x)$ in \eqref{eq:ub def}, it then follows that 
$$F(\tilde{\Omega})\subseteq \left(\{\ell b\} + \R^m_+\right)\cap \left(\{\bar{z} + e \}- \R^m_+\right),$$ which is a bounded set. Thus, $F(\tilde{\Omega})$ is bounded.

  In order to show \ref{item:F closed graph}, i.e.\ the closedness of $\gph F,$ we fix a sequence $\{(x^k, y^k)\}_{k \in \N} \subseteq \gph F$ such that $\lim_{k \to \infty} (x^k, y^k) = (\bar{x}, \bar{y}).$ Then, by the definition of $F$, we have for all $ k \in \N$ that $ y^k \in [\ell b, ub(x^k)]_C.$ By Proposition \ref{prop: cl A-K equivalent}, we further have for all $ k \in \N$ that $y^k\not\in F_{\mathcal{U}}(x^k)- \Int \R^m_+$ and thus for all $ u \in \mathcal{U}$ that $   y^k\nless f(x^k, u) $. Now, since $\Omega$ is closed, in particular $\bar{x} \in \Omega.$ Moreover, because of Berge's Maximum Theorem \cite[Theorem 17.31]{Aliprantis2006}, we  have that $ub$ is continuous. This, combined with the closedness of $C$, allows us to take the limit in $ y^k \in [\ell b, ub(x^k)]_C$ for $k \to \infty$ to obtain 
\begin{equation}\label{eq:ybar in box}
    \bar{y} \in [\ell b, ub(\bar{x})]_C.
\end{equation}
Now we consider $y^k \nless f(x^k, u)$ for all $ u \in \mathcal{U}$. First, fix some $u\in\mathcal{U}$. We have for all $k\in\N$ that there exists at least one $i\in[m]$  with $f_i(x^k,u)\leq y_i^k$. By taking suitable subsequences and the continuity of $f$ into account, we obtain for at least one $i\in[m]$ that $f_i(\bar x,u) \leq \bar y_i$. As a consequence, for all $u\in \mathcal{U}$, we have $\bar y\not\in F_{\mathcal{U}}(\bar x)-\Int \R^m_+$. 
According to Proposition \ref{prop: cl A-K equivalent}, this is now equivalent to $\bar{y} \in \cl(F_{\mathcal{U}}(\bar x) - \R_+^m)^c.$ From this and  \eqref{eq:ybar in box}, it then follows that $\bar{y} \in F(\bar{x}),$ which shows the closedness of $F$. 

To see \ref{item:FtoFUweak} note that from the continuity of $f$ and the compactness of $\mathcal{U},$ it is clear that $F_\mathcal{U}$ is compact-valued. The statement is then an immediate consequence of \eqref{eq:F_UsubsetyU}  and Theorem \ref{thm:uppertolower}. 
 
 In order to see \ref{item:FtoFU}, fix $x, x' \in \Omega$ such that $F(x) \preceq^l F(x')$. Suppose now that $ F_{\mathcal{U}}(x) \npreceq^u F_{\mathcal{U}}(x').$ Then, there exists $y \in F_{\mathcal{U}}(x) \cap \left(F_{\mathcal{U}}(x') - \R^m_+\right)^c.$ Therefore, according to \eqref{eq:F_UsubsetyU}, we have that $y \in \left(\{\ell b\} + \Int C\right) \cap \left(F_{\mathcal{U}}(x') - \R^m_+\right)^c,$ an open set. Consequently, we can find some small $t > 0$ such that $y - t e \in  \left(\{\ell b \}+ C\right) \cap \cl\left(F_{\mathcal{U}}(x') - \R^m_+\right)^c.$ It follows now from Lemma \ref{lem:shiftinglemma} the existence of $y' \in F(x')$ such that $y' \leq y - te.$ Moreover, since $F(x) \preceq^l F(x'),$ there exists $z \in F(x)$ such that $z \leq y'.$ Thus, in particular, $z \leq y - te.$ However, this would imply that $$z \in F_{\mathcal{U}}(x) - \{t e \}- \R^m_+ \subseteq F_{\mathcal{U}}(x) - \Int \R^m_+,$$  which, taking into account Proposition \ref{prop: cl A-K equivalent},  contradicts the fact that $z \in F(x)$. The statement follows.
\end{proof}

The following proposition shows an important connection between \eqref{sp} (from now on with $F$ given by \eqref{eq:Fhood}) and the robust weakly efficient solutions of \eqref{ump}.

\begin{Proposition}\label{prop:spu=spl}
 Let Assumption \ref{main_assumption} hold and $F$ be   given by \eqref{eq:Fhood}. Then it holds that  $$\wargmin \eqref{ump} = \wargmin^l \eqref{sp}.$$
\end{Proposition}
\begin{proof}
It is an immediate consequence of Proposition \ref{prop:basic properties F} \ref{item:FtoFUweak}.
\end{proof}

Proposition \ref{prop:spu=spl} opens the door to apply the results from Theorem \ref{thm:splmainresults} to our context. The first step in this direction is the following lemma.
\begin{Lemma}\label{lem:clFu-Kc}
Let Assumption \ref{main_assumption} hold. Then,
\begin{displaymath} 
\forall\; x \in \Omega: \cl \left(F_\mathcal{U}(x)- \R^m_+\right)^c = \left\{y \in \R^m \ \Bigg|\    \min\limits_{j \in [m]}\left\{f_j(x,u) - y_j\right\} \leq 0 \;\; \forall \;u \in \mathcal{U}\right\}.
\end{displaymath}
\end{Lemma}
\begin{proof} Fix $x \in \Omega.$ Then, taking into account that $F_\mathcal{U}(x)$ is compact, Lemma \ref{lem:psiB}  \ref{item:representationclB-Kc}, \eqref{eq:psiB} and the definition of $F_{\mathcal{U}}$ in  \eqref{eq:F_U}, we deduce
\begin{eqnarray*}
\cl \left(F_\mathcal{U}(x)- \R^m_+\right)^c & = & \left\{y \in \R^m \mid - \psi_{F_{\mathcal{U}}(x)}(y) \leq 0 \right\}\\
                            & = & \left\{y \in \R^m \mid - \min_{y' \in F_{\mathcal{U}}(x)}\psi (y -y') \leq 0  \right\}\\
                            & = & \left\{y \in \R^m \mid - \min_{u \in \mathcal{U}}\psi (y -f(x,u)) \leq 0  \right\}\\
                            & = & \left\{y \in \R^m \mid \max\limits_{u \in \mathcal{U}} -\psi\left( y - f(x,u)\right) \leq 0 \right\}\\
                            & = & \left\{y \in \R^m \mid  -\psi\left( y - f(x,u)\right) \leq 0 \;\; \forall \;u \in \mathcal{U}\right\},
\end{eqnarray*} as desired.
\end{proof}

For our analysis in the rest of the section, we will follow the notation employed in Theorem \ref{thm:splmainresults} for the projections of the sets $\varepsilon \textrm{-} \wargmin \eqref{mp}$ onto $\R^n,$ i.e., 
$$\varepsilon\textrm{-}\wargmin_x\eqref{mp}:= \{x' \in \R^n \mid \exists\; y^1,\ldots y^p \in F(x') : (x',y^1,\ldots, y^p) \in \varepsilon\textrm{-}\wargmin \eqref{mp}\}.$$ The next lemma shows that the $\varepsilon$-weakly efficient solutions (projected onto $\R^n$) of the multiobjective optimization problems \eqref{mp} (our epigraphical reformulation) and \eqref{vpp} (from the vectorization scheme) coincide. 

\begin{Lemma}\label{lem:mop=mop'}
Let Assumption \ref{main_assumption} hold. For $p \in \N,$ consider the multiobjective optimization  problem \eqref{vpp} associated to the set-valued mapping $F$ given by \eqref{eq:Fhood}. Then,
$$ \forall\; \varepsilon \geq 0: \varepsilon \textrm{-} \wargmin_x\eqref{mp} = \varepsilon \textrm{-} \wargmin_x\eqref{vpp}.$$
\end{Lemma}
\begin{proof} Because of the particular structure of the set-valued mapping $F$, we can apply Lemma \ref{lem:clFu-Kc}  to obtain that \eqref{vpp} can be rewritten as
\begin{equation}\label{mp'}
\begin{array}{ll}
\min\limits_{x,y^1, \ldots,y^p} \; \begin{pmatrix}
y^1 \\ \vdots \\y^p
\end{pmatrix}   \tag{$\mathcal{MP}^{'}_p$} \\
  \\
\quad\textup{s.t.} \;\; \min\limits_{j \in [m]}\left\{f_j(x,u) - y^i_j\right\} \leq 0 \;\; \forall \;u \in \mathcal{U}, \; i \in [p],  \\
\quad \;\;\;\;\;\;\;\;\;  y^i \in \left[\ell b, ub(x) \right]_C, \; \forall\; i \in [p],\\ 
\quad \;\;\;\;\;\;\;\;\; x \in \Omega.
\end{array} 
\end{equation} 
The only difference to \eqref{mp} are the additional constraints $y^i\leq ub(x)$ for $i\in[p]$. 

Now, let $\bar{x} \in \varepsilon \textrm{-}\wargmin_x\eqref{mp},$ and choose $\left\{\bar{y}^1,\ldots,\bar{y}^p\right\} \subseteq \R^m$ such that 
\begin{equation}\label{eq:wargminmop}
\left(\bar{x},\bar{y}^1,\ldots,\bar{y}^p\right) \in \varepsilon \textrm{-}\wargmin\eqref{mp}.
\end{equation} Taking into account Lemma \ref{lem:clFu-Kc},   the feasibility of $\left(\bar{x},\bar{y}^1,\ldots,\bar{y}^p\right)$ for \eqref{mp} is equivalent to 
$$\forall \; i \in [p]: \bar{y}^i \in \cl \left(F_\mathcal{U}(x)- \R^m_+\right)^c\cap \left(\{\ell b\}+C\right).$$ Applying now Lemma \ref{lem:shiftinglemma}, we obtain for all $i\in[p]$ that $\bar y^i=\bar z^i+\bar k^i$ with $\bar k^i\in\R^m_+$ and
$$\bar z^i  \in \cl \left(F_\mathcal{U}(\bar x)- \R^m_+\right)^c\cap [\ell b,ub(\bar x)]_C.$$
In particular,  for all $i\in[p]$  we have $\bar z^i\leq ub(\bar x)$ and $\left(\bar{x},\bar{z}^1,\ldots,\bar{z}^p\right)$ is feasible for \eqref{mp'}.
It suffices now to show that $\left(\bar{x},\bar{z}^1,\ldots,\bar{z}^p\right) \in \varepsilon \textrm{-}\wargmin \eqref{mp'}.$ Indeed, otherwise we can find $\left(x,y^1,\ldots,y^p\right)$ feasible for \eqref{mp'} (and thus feasible for \eqref{mp}) such that 
$$\forall \, i\in [p]: y^i < \bar{z}^i- \varepsilon e\leq \bar y^i- \varepsilon e.$$ This would then contradict \eqref{eq:wargminmop}.

Suppose now that $\bar{x} \in \varepsilon \textrm{-} \wargmin_x\eqref{mp'},$ and take elements $\bar{y}^1,\ldots,\bar{y} ^p\in \R^m$ such that 
\begin{equation}\label{eq:wargminmop'}
\left(\bar{x},\bar{y}^1,\ldots,\bar{y}^p\right) \in \varepsilon \textrm{-}\wargmin\eqref{mp'}.
\end{equation} We claim that $\left(\bar{x},\bar{y}^1,\ldots,\bar{y}^p\right) \in \varepsilon \textrm{-}\wargmin \eqref{mp}.$ Indeed, otherwise we can find $\left(x,y^1,\ldots,y^p\right)$ feasible for \eqref{mp} such that 
\begin{equation}\label{eq:yilessybari}
\forall \, i\in [p]: y^i < \bar{y}^i-\varepsilon e.
\end{equation} Similarly to above, we can apply Lemma \ref{lem:clFu-Kc} to obtain 
$$\forall \; i \in [p]: y^i \in \cl \left(F_\mathcal{U}(x)- \R^m_+\right)^c\cap \left(\ell b+C\right).$$ Applying next Lemma \ref{lem:shiftinglemma}, we get for $i\in[p]$, as above, $z^i\in\R^m$  with $ z^i\leq y^i$ and $z^i\leq ub(\bar x)$  and 
$  z^i \in \cl \left(F_\mathcal{U}(x)- \R^m_+\right)^c\cap \left(\ell b+C\right)$.
%we can assume without loss of generality that 
%$$\forall \; i \in [p]: y^i \leq ub(x).$$ 
This would then imply that $\left(x,z^1,\ldots,z^p\right)$  is feasible for \eqref{mp'}. This, together with \eqref{eq:yilessybari} and $z^i\leq y^i$,  implies for all $i\in[p]$ that $z^i<\bar y^i-\varepsilon e$, which contradicts  \eqref{eq:wargminmop'}.
\end{proof}

We can now establish the main results of this section for our epigraphical reformulation. They are applications of Theorems \ref{thm:splmainresults} and  \ref{thm:sufficient vectorization}, respectively, which have been made possible by the preceding results from this section.  
\begin{Theorem}\label{thm:main approximation property}
 Let Assumption \ref{main_assumption} hold. Then,
\begin{enumerate}
\item \label{item:monot solution sets} For every $\varepsilon \geq 0$ and $p_1, p_2 \in \N$ we have: 
$$p_1 \leq p_2 \Longrightarrow \varepsilon \textrm{-} \wargmin_x \;(\mathcal{MP}_{p_1}) \subseteq \varepsilon \textrm{-} \wargmin_x \;(\mathcal{MP}_{p_2}).$$ 
\item \label{item:approximation property} The following approximation property is true:
\begin{displaymath}
\bigcup\limits_{p \in \N} \wargmin_x \eqref{mp} \subseteq \wargmin \eqref{ump} = \bigcap\limits_{\varepsilon >0}\bigcup\limits_{p\in \N} \varepsilon \textrm{-}\wargmin_x  \eqref{mp}.
\end{displaymath} Furthermore, if $f\left(\wargmin \eqref{ump} \times \mathcal{U}\right)$ is bounded, then for each $\varepsilon>0$ we can find $p \in \N$ such that
\begin{displaymath}
\wargmin_x \eqref{mp} \subseteq \wargmin \eqref{ump} \subseteq \varepsilon \textrm{-} \wargmin_x \eqref{mp}.
\end{displaymath}
\end{enumerate}
\end{Theorem}
\begin{proof} Statement $(i)$  follows directly from combining  Lemma \ref{lem:mop=mop'} and  Theorem \ref{thm:splmainresults} \ref{item:monotwargminvpp}.
%
%For $p_1\leq p_2,$ we can apply Lemma \ref{lem:mop=mop'}, Theorem \ref{thm:splmainresults} \ref{item:monotwargminvpp} and again Lemma \ref{lem:mop=mop'} to obtain
%\begin{eqnarray*}
%\varepsilon \textrm{-}\wargmin_x  (\mathcal{MP}_{p_1}) & = & \varepsilon \textrm{-}\wargmin_x  (\mathcal{VP}_{p_1})\\
%                                         & \subseteq & \varepsilon \textrm{-}\wargmin_x  (\mathcal{VP}_{p_2})\\
%                                         & = & \varepsilon \textrm{-}\wargmin_x  (\mathcal{MP}_{p_2}).
%\end{eqnarray*}
%
The first part of statement $(ii)$ follows by combining Lemma \ref{lem:mop=mop'}, Theorem \ref{thm:splmainresults} \ref{item:approxpropertyvpp} and Proposition \ref{prop:spu=spl}.
%Applying consecutively Lemma \ref{lem:mop=mop'}, Theorem \ref{thm:splmainresults} \ref{item:approxpropertyvpp}, Proposition \ref{prop:spu=spl}, Proposition \ref{prop:spu=spl}, Theorem \ref{thm:splmainresults} and Lemma \ref{lem:mop=mop'}, we deduce
%
%\begin{eqnarray*}
%\bigcup\limits_{p\in \N} \wargmin_x \eqref{mp} & = & \bigcup\limits_{p\in \N} \wargmin_x \eqref{vpp}\\
%         & \subseteq & \wargmin^l \eqref{sp}\\
 %                                     %          & = & \wargmin \eqref{ump}\\
%                                                & = & \wargmin^l \eqref{sp}\\
 %                                               & = & \bigcap\limits_{\varepsilon >0} \bigcup\limits_{p \in \N} \varepsilon \textrm{-} \wargmin_x \eqref{vpp}\\
 %                                               & = & \bigcap\limits_{\varepsilon >0} \bigcup\limits_{p \in \N} \varepsilon \textrm{-} \wargmin_x \eqref{mp}.
%\end{eqnarray*}
%In order to see 
The second part of $(ii)$ follows by Proposition \ref{prop:basic properties F} \ref{item: F(Omega) bounded},  Theorem \ref{thm:splmainresults} \ref{item:approxpropertyvpp}, Proposition \ref{prop:spu=spl} and Lemma \ref{lem:mop=mop'}.
%
%note that from the boundedness of the set $f\left(\wargmin \eqref{ump} \times \mathcal{U}\right)$ and Proposition \ref{prop:basic properties F} \ref{item: F(Omega) bounded}, we have that $F(\Omega)$ is bounded. Thus, the additional condition in the second part of Theorem \ref{thm:splmainresults} \ref{item:approxpropertyvpp} is also fulfilled. Taking this into account, together with Proposition \ref{prop:spu=spl} and Lemma \ref{lem:mop=mop'}, the statement follows. 
\end{proof}

\begin{Remark}
Theorem \ref{thm:main approximation property} \ref{item:approximation property} together with the first inclusion in Theorem \ref{thm:main approximation property} \ref{item:monot solution sets} shows that the sets $\left\{\wargmin_x \eqref{mp}\right\}_{p \in \N}$ form a parametric family of inner approximations of $\wargmin \eqref{ump}$. The quality of the approximation improves for higher values of $p.$ On the other hand, the equality in the right hand side of Theorem \ref{thm:main approximation property} \ref{item:approximation property} indicates that, if  all $\varepsilon$-weakly efficient solutions can be computed, then every robust weakly efficient solution of \eqref{ump} can be captured if we solve \eqref{mp} for $p$ high enough. Under additional assumptions, the second part of Theorem \ref{thm:main approximation property}  \ref{item:approximation property} then establishes that this value of $p$ can be chosen independently of the robust solution being considered.
\end{Remark}

Our last result shows that it is possible to focus only on computing robust weakly efficient solutions  of \eqref{ump} that come from the epigraphical reformulations. The reason is that, image space-wise, no quality solutions are being lost by this procedure.

\begin{Theorem}
Let Assumption \ref{main_assumption} hold and suppose in addition that $\Omega$ is compact. Then,

\begin{equation*}
    \forall\; x\in \Omega, \; \exists\; \bar{x} \in \cl\left(\bigcup_{p \in \N} \wargmin_x\eqref{mp}\right): F_{\mathcal{U}}(\bar{x}) \preceq^u F_{\mathcal{U}}(x). 
\end{equation*}
\end{Theorem}
\begin{proof}
Because of the continuity of $f$ and the compactness of $\Omega \times \mathcal{U}$, we have in particular that $f(\Omega \times \mathcal{U})$ is bounded. Therefore, according Proposition \ref{prop:basic properties F} \ref{item: F(Omega) bounded}, $F(\Omega)$ is bounded. Moreover, from Proposition \ref{prop:basic properties F} \ref{item:F closed graph} we have that $\gph F$ is closed. Thus, we can apply Theorem \ref{thm:sufficient vectorization} and take into account Lemma \ref{lem:mop=mop'} to obtain
\begin{equation}\label{eq:suff spl aux2}
\forall\; x\in \Omega, \exists\; \bar{x} \in \cl \left(\bigcup\limits_{p \in \N} \wargmin_x  \eqref{mp}\right) : F(\bar{x})\preceq^l F(x).
\end{equation} The statement thus follows from \eqref{eq:suff spl aux2} and Proposition \ref{prop:basic properties F} \ref{item:FtoFU}.
\end{proof}

\subsection{Exactness of the Approximation}

Under Assumption \ref{main_assumption} we have that for any $p\in\N$ it holds that 
$\wargmin_x \eqref{mp} \subseteq \wargmin \eqref{ump}$.  In this part of the section we show that in certain cases there is even a finite $p$ for which equality holds. This means that in these cases all robust weakly efficient solution can be found by solving a certain multiobjective optimization problem. 
\begin{Definition}
Let Assumption \ref{main_assumption} hold. We say that \eqref{ump} satisfies the weakly efficient finite dimensional vectorization property \textup{(wFDVP)} if there exists some $p\in \N$ such that
$$\wargmin_x \eqref{mp} = \wargmin \eqref{ump}.$$ 
\end{Definition}

In the scalar case, i.e., when $m = 1$, \eqref{ump} satisfies \textup{(wFDVP)} with $p = 1$ as in this case the epigraphical reformulation \eqref{eq:epi ref scalar} is equivalent to \eqref{rc scalar}. However, for uncertain multiobjective problems, considering just $p=1$ is in general not sufficient which motivates our definition above.

First, for our forthcoming proofs, we need to construct a polyhedral cone which replaces the cone $C$ from Assumption \ref{main_assumption}. For $\alpha > m - 1,$ we consider the cone 
\begin{equation}\label{eq:C^alpha}
    C^\alpha := \{y \in \R^m \mid M^\alpha y \leq 0\},
\end{equation} where
$$M^\alpha := \begin{pmatrix}
- \alpha &1 & \ldots & 1& 1\\
1& -\alpha & \ldots  & 1&1\\
\vdots & \vdots& \ddots & \vdots& \vdots\\
1   & 1 & \ldots &  -\alpha & 1 \\
1 &1 &  \ldots  & 1 &- \alpha
\end{pmatrix}\in\R^{m\times m}.$$
For instance, for $m=2$ and $\alpha >2$ we get the cones
$C^\alpha=\{y\in\R^2\mid \frac{1}{\alpha}y_1\leq y_2\leq \alpha y_1\}.$

\begin{Proposition}\label{prop: Calpha properties}
The following statements hold:
\begin{enumerate}
    \item For any $\alpha > m - 1,$ the set $C^\alpha$ is a closed, convex, pointed and solid cone.
    \item \label{item:calphaunionpareto}$\bigcup_{\alpha > m - 1} C^\alpha = \Int \R_{+}^m \cup \{0\}.$
    \item The family $\{C^\alpha\}_{\alpha  > m - 1 }$ is monotonically  increasing in $\alpha $ with respect to inclusion, i.e., $\alpha_1 \leq \alpha_2 \Longrightarrow C^{\alpha_1} \subseteq C^{\alpha_2}.$
    \item \label{item:CsubsetCalpha} If $C \subseteq \Int \R_{+}^m \cup \{0\}$ is a closed cone, then there exists $\alpha \in \R$ such that $C \subseteq C^\alpha.$ 
\end{enumerate}
\end{Proposition}
\begin{proof}
$(i)$ The closedness, convexity, and pointedness of $C^\alpha$ are clear from the definition. The solidness is an immediate consequence of the fact that $M^\alpha e < 0,$ which implies $e \in \Int C^\alpha.$

$(ii)$ Let $\alpha > m - 1$ and $y \in C^\alpha.$ Then, 
\begin{equation}\label{eq:yinCalpha}
    \forall\; i \in [m]: \sum_{j\neq i} y_j \leq \alpha y_i.
\end{equation} Adding these inequalities over $i \in [m]$ we  get $(m - 1) \sum_{j = 1}^m y_j \leq \alpha\sum_{i = 1}^m y_i,$ or equivalently, $0 \leq (\alpha - (m - 1))\sum_{j = 1}^m y_j.$ This, together with the fact that $\alpha > m - 1,$ implies for all $i\in[m]$
\begin{equation*}
    0 \leq \sum_{j = 1}^m y_j = y_i + \sum_{j \neq i} y_j \leq  y_i+\alpha y_i=(\alpha + 1)y_i. 
\end{equation*} Thus, $y \geq 0.$ Suppose that $y_i = 0$ for some $i \in [m].$ Then, from \eqref{eq:yinCalpha} we find that $\sum_{j \neq i} y_j \leq \alpha y_i = 0$ and therefore $y = 0.$ This shows that $y \in \Int \R_{+}^m \cup \{0\}.$  

Take now any $y \in \Int \R_{+}^m \cup \{0\}.$ Note that if $y = 0$ there is nothing to prove. Thus, suppose that $y > 0$ and define 
$$\alpha := \max\left\{m\,,\, \max_{i\in [m]}\left\{\frac{\sum\limits_{j \neq i}y_j}{y_i}\right\} \right\}.$$ It is then straightforward to check that \eqref{eq:yinCalpha} holds, and therefore $y \in C^\alpha.$

$(iii)$ This follows   from %\ref{item:calphaunionpareto} and 
\eqref{eq:yinCalpha}.

$(iv)$ Assume this is not true. Then, there exists a sequence $\{c^k\}_{k\in \N} \subseteq C$ such that 
$\|c^k\| = 1$ and $c^k \notin C^{m + k}$ for each $k \in \N.$ In particular $c^k\in \Int \R^m_+$ for all $k$. Since $\{c^k\}_{k\in \N} $ is bounded, we can assume without loss of generality that it is convergent. Let $\bar{c}:= \lim_{k\to \infty} c^k.$ Then, by the closedness of $C,$  we have that $\bar{c} \in C.$ Further, $\bar c\neq 0$ as $\|\bar c\|=1$ and hence, in particular, $\bar c\in\Int\R^m_+$.
On the other hand, since $c^k \notin C^{m + k}$ for every $k \in \N,$ we can assume, without loss of generality, the existence of $i \in [m]$ such that 
\begin{equation}\label{eq:opositeineqC}
    \forall\; k \in \N:\ \sum\limits_{j \neq i}c^k_j
    >\alpha c^k_i>(m+k)c_i^k .
\end{equation} Then, taking the limit in \eqref{eq:opositeineqC} for $k \to \infty$ we obtain, as $\bar c_i>0$, 
\begin{equation*}
\sum\limits_{j \neq i}\bar c_j
    \geq \lim_{k\to\infty}  (m+k)c_i^k =\infty,   
\end{equation*} a contradiction.
\end{proof}

The important consequence of Proposition \ref{prop: Calpha properties} (iv) is that we can restrict ourselves to consider proper boundedness of functions only with respect to cones in the parametric family $\{C^\alpha\}_{\alpha > m - 1}:$
\begin{Proposition} \label{prop:Calpha is enough}
Let Assumption \ref{basic_setup} hold. Then, $f$ is properly lower bounded on $\Omega \times \mathcal{U}$ if and only if there exists $\ell b \in \R^m$ and $\alpha > m - 1$ such that $f$ is properly lower bounded on $\Omega \times \mathcal{U}$ with respect to $\ell b$ and $C^\alpha.$
\end{Proposition}

We can now make use of the special structure of $C^\alpha$. The next lemma prepares a result which is required for the proof of   Lemma \ref{lem:finitelocallowerbounds}. That again is needed for applying 
Theorem \ref{thm:splmainresults} \ref{item:fdvpvpp} to our setting.

\begin{Lemma}\label{lem:stronglyefficientmp}
For $\alpha > m - 1,$ consider the cone $C^\alpha$ given by \eqref{eq:C^alpha}. Let $a,b \in \R^m$ be such that $a < b,$ and consider elements $\ell, \bar{z} \in B:= (\{a\} + C^\alpha) \cap (\{b \}- \R_+^m)$ such that $\ell < \bar{z}.$ Then, for each $i \in [m],$ the multiobjective optimization problem
\begin{equation}\label{plzi}
\begin{array}{rl}
 & \min\limits_y y   \tag{$\mathcal{P}_i(\ell,\bar{z})$} \\
                        & \; \;\textup{s.t.} \;\; y_i = \bar{z}_i,\\
                        & \qquad \; \; y \in B \cap (\{\ell\} + \R_+^m)
\end{array} 
\end{equation} has a strongly efficient solution.
\end{Lemma}
\begin{proof}
Fix $i \in [m].$ Then, taking into account the structure of $M^\alpha,$ \eqref{plzi} takes the form 
\begin{equation}
\begin{array}{rl}
 & \min\limits_y y   \tag{$\mathcal{P}_i(\ell,\bar{z})$} \\
                        & \; \;\textup{s.t.} \;\; y_i = \bar{z}_i,\\
                        & \qquad \; \; \sum\limits_{j \neq k}(y_j - a_j) \leq \alpha(y_k - a_k) \;\forall\  k\in [m],\\
                        & \qquad \; \; \ell \leq y\leq b. \\
                       % & \qquad \; \; y \leq b.
                       \end{array} 
\end{equation} Consider now the mapping $T: \R^m \rightarrow \R^m$ defined by 
\begin{equation*}
    \forall \; y \in \R^m, k \in [m]: \left(T(y)\right)_k:= \left\{
\begin{array}{ll}
      \max\left\{\ell_k\,,\, a_k + \frac{\sum\limits_{j \neq k}(y_j - a_j)}{\alpha}\right\} &  \textrm{ if } k \neq i ,\\
      \bar{z}_i & \textrm{ if } k = i .\\
\end{array} 
\right. 
\end{equation*} We divide the rest of the proof in several steps.
First, as step 1, we %\underline{\textbf{Step 1:}} We 
show that 
\begin{equation*}
     y \textrm{ feasible for } \eqref{plzi} \Longrightarrow T(y) \textrm{ feasible for } \eqref{plzi} \textrm{ and }T(y) \leq y.
\end{equation*} We start by proving the inequality 
\begin{equation}\label{eq:Tyleqy}
    T(y) \leq y.
\end{equation}  Indeed, because of the feasibility of $y$ for \eqref{plzi}, for $k = i$ we have $\left(T(y)\right)_k = \bar{z}_i = y_k.$ Similarly, the feasibility of $y$ for \eqref{plzi}  implies
\begin{equation}\label{eq:implyfeasible}
  \forall\; k\in [m]\setminus\{i\}: \sum\limits_{j\neq k}(y_j - a_j) \leq \alpha(y_k - a_k) \textrm{ and } \ell_k \leq y_k.
\end{equation} From \eqref{eq:implyfeasible} it  follows that 
\[ 
    \forall\; k\in [m]\setminus\{i\}: y_k \geq \max\left\{\ell_k, a_k + \frac{\sum\limits_{j \neq k}(y_j - a_j)}{\alpha}\right\} = \left(T(y)\right)_k. 
\] This proves \eqref{eq:Tyleqy}.

Next, we show that $T(y)$ is feasible for \eqref{plzi}. Indeed, by the definition of $T,$ we have $\left(T(y)\right)_i = \bar{z}_i$ and $\ell \leq T(y).$ Moreover, from \eqref{eq:Tyleqy} it follows $T(y) \leq y \leq b.$ It thus remains to show that 
\begin{equation}\label{eq:coneineq}
   \forall\; k \in [m]: \sum\limits_{j \neq k}((T(y))_j - a_j) \leq \alpha((T(y))_k - a_k).
\end{equation} In order to see this note that, for $k = i,$ the inequality \eqref{eq:Tyleqy} and the feasibility of $y$ for \eqref{plzi} imply 
$$\sum\limits_{j \neq k}((T(y))_j - a_j) \leq \sum\limits_{j \neq k}(y_j - a_j) \leq \alpha (y_i - a_i) = \alpha (\bar{z}_i - a_i) = ((T(y))_k - a_k).$$ Similarly, \eqref{eq:Tyleqy} implies for all $k \in [m] \setminus\{i\}$
$$\begin{array}{rcll}
%\forall\; k \in [m] \setminus\{i\}: 
\sum\limits_{j \neq k}((T(y))_j - a_j) & \leq &  \sum\limits_{j \neq k}(y_j - a_j)  
  %   & = & \alpha a_k + \sum\limits_{j \neq k}(y_j - a_j)  - \alpha a_k\\
   =  \alpha \left(a_k + \frac{\sum\limits_{j \neq k}(y_j - a_j)}{\alpha} -  a_k\right)\\
     & \leq & \alpha \left(\max\left\{\ell_k, a_k + \frac{\sum\limits_{j \neq k}(y_j - a_j)}{\alpha}\right\} -  a_k\right)
      &=  \alpha ((T(y))_k - a_k).
\end{array} $$
The feasibility of $T(y)$ follows. 

\medskip

Second, as step 2, we show that the mapping $T$ has a unique fixed point $\ell^i$.  
This will be an immediate consequence of Banach's fixed point theorem, see, for example, \cite[Theorem 5.1-2]{kreyszig89}. Thus, we proceed to show that $T$ is a contraction, i.e., there exists a constant $\gamma < 1$ such that for all 
$%\begin{equation}\label{eq:contractiondef}
     y^1, y^2 \in \R^m$ we have $ \|T(y^1) - T(y^2)\|_1 \leq \gamma \|y^1 - y^2\|_1.
$ %\end{equation} 
Indeed, fix $y^1, y^2 \in \R^m.$ Then, 
$$ \begin{array}{rcl}
\|T(y^1) - T(y^2)\|_1 & = & \sum\limits_{k\neq i} \left|\max\left\{\ell_k, a_k + \frac{\sum\limits_{j\neq k} (y^1_j - a_j)}{\alpha}\right\} - \max\left\{\ell_k, a_k + \frac{\sum\limits_{j\neq k} (y^2_j - a_j)}{\alpha}\right\}\right|\\
                     & \leq & \frac{1}{\alpha} \sum\limits_{k \neq i} \left| \sum\limits_{j \neq k} (y^1_j - y^2_j)\right| 
                      \leq  \frac{1}{\alpha} \sum\limits_{k \neq i} \sum\limits_{j \neq k} |y^1_j - y^2_j|\\
                    % & \leq & \frac{1}{\alpha} \sum\limits_{k \neq i} \sum\limits_{j = 1}^m |y^1_j - y^2_j|\\
                     & = & \frac{1}{\alpha} \sum\limits_{k \neq i} \|y^1 - y^2\|_1 
                       =   \left(\frac{m-1}{\alpha}\right)\|y^1 - y^2\|_1.
\end{array} $$
Therefore, taking into account that $\alpha > m - 1,$  the mapping $T$ is a contraction with $\gamma :=  \frac{m-1}{\alpha} .$ Applying now Banach's fixed point theorem, we obtain the desired point $\ell^i.$ 

\medskip

Finally, as step 3, we show that $\ell^i$ is a strong minimizer of \eqref{plzi}. 
Indeed, as a consequence of Banach's fixed point theorem, the feasibility of $\bar{z}$ for \eqref{plzi} and the statement proved in Step 1, we have that:
\begin{itemize}
    \item $\ell^i = \lim\limits_{r \to \infty} T^{(r)}(\bar{z}),$ where we denote with $T^{(r)}$ the $r$-times iterative application of the mapping $T$,
    \item every point in the sequence $\left\{T^{(r)}(\bar{z})\right\}_{r \in \N}$ is feasible for \eqref{plzi}.
\end{itemize} In particular, due to the closedness of the feasible region of \eqref{plzi}, we find that $\ell^i$ is also feasible for this problem. Take now a feasible point $y.$ Then, again by Banach's fixed point theorem and the statement in Step 1 of this proof, we find that $\ell^i = \lim\limits_{r \to \infty} T^{(r)}(y) \leq y.$ The strong efficiency of $\ell^i$ for \eqref{plzi} follows.
\end{proof}

As already mentioned, for applying 
Theorem \ref{thm:splmainresults} \ref{item:fdvpvpp}, we need the next lemma before we can formulate our main results on the exactness of the approximation. 

\begin{Lemma}\label{lem:finitelocallowerbounds}
Let $s \in \N \cup \{0\}$ and $\alpha > m -1$ be given. For $a, b \in \R^m,$ consider the set $B:= (\{a \}+ C^\alpha) \cap (\{b\} -\R_+^m).$  Then, for any $S \subseteq \Int B$ such that $|S| = s,$ the set 

\[A(S) := \left\{
\begin{array}{ll}
      \{a\}, & \textrm{if } s = 0,\\
       \bigcap_{z \in S} (\{z\} - \Int \R_{+}^m)^c \cap B, & \textrm{otherwise} \\

\end{array} 
\right. \] satisfies 

$$|\Min(A(S), \R_+^m)| \leq m^s.$$
\end{Lemma}
\begin{proof}
We argue by induction on $s.$ For $s = 0$ we have indeed $|\Min(A(\emptyset), \R_+^m)|  = 1.$ Suppose now that for some $s \in \N$ the statement holds and let $S \subseteq \Int B$ be such that  $|S| = s + 1.$ Fix $\bar{z} \in S.$ Then, $|S\setminus \{\bar{z}\}| = s$ and, according to the induction hypothesis, we have 
$$|\Min(A(S\setminus \{\bar{z}\}), \R_+^m)| \leq m^s.$$ Define $D:= \{\ell \in \Min(A(S\setminus \{\bar{z}\}), \R_+^m) \mid \ell < \bar{z}\}$ and, for every $\ell \in D$ and each $i \in [m],$ let us denote by $\ell^i$ the strongly efficient solution of the problem \eqref{plzi}. Note that $\ell^i$ is well defined because of Lemma \ref{lem:stronglyefficientmp}. Set now
$$P := \bigcup_{\ell \in D}\{\ell^i \mid i \in [m]\}.$$ %Our argument will be a consequence of the following claims:
Our next argumentation is based on the following three statements.

First, we have  $P \subseteq A(S).$    
    Indeed, otherwise there exists $\ell \in D$ and $i\in [m]$ such that $\ell^i \notin A(S).$ Thus, either $\ell^i \notin B$ or there exists $z \in S$ such that $\ell^i < z.$ However, since $\ell^i$ is feasible for \eqref{plzi}, we have in particular that $\ell^i \in B.$ Thus, it can only be that $\ell^i < z$ for some $z \in S.$ Note that, because $\ell^i$ is feasible for \eqref{plzi}, we also have $\ell^i_i = \bar{z_i}.$ Therefore, $z \in S \setminus\{\bar{z}\}.$  But then, again by the feasibility of $\ell^i$ for \eqref{plzi}, we find that 
    $\ell \leq \ell^i < z,$ a contradiction to the fact that $\ell \in A(S\setminus \{\bar{z}\}).$
    
Second, we show  the inclusion $\Min(A(S\setminus \{\bar{z}\}), \R_+^m) \setminus D \subseteq A(S).$ Therefore,  take any $\ell \in \Min(A(S\setminus \{\bar{z}\}), \R_+^m).$ Then, by definition, we have $\ell \in B$ and $\ell \not< z$  for every $z \in S \setminus\{\bar{z}\}.$ Moreover, since $\ell \notin D,$ it also follows that $\ell \not< \bar{z}.$ Thus,  $\ell \in A.$
    
Finally, as third statement, we prove  $A(S) \subseteq \left(P \cup \left(\Min(A(S\setminus \{\bar{z}\}), \R_+^m) \setminus D\right) \right)+ \R_+^m.$
Let $y \in A(S).$ 
Then, in particular, $y \in A(S \setminus\{\bar{z}\}).$ As external stability holds, see \cite[Theorem 3.2.10]{Sawaragi}, there exists  $\ell \in \Min(A(S \setminus\{\bar{z}\}), \R_+^m),$ with $\ell \leq y.$ Suppose first that $\ell \not< \bar{z}.$ Then, by definition $\ell \notin D,$ and there is nothing to prove. Assume now that $\ell < \bar{z}$ and, for $t \in [0,1],$ set     $$y^t:= t \ell + (1 -t) y.$$ Then, because $y$ and $\ell$ belong to the convex set $B\cap (\{\ell\} + \R_+^m),$ it follows that $y^t \in B\cap (\{\ell\} + \R_+^m)$ for every $t\in [0,1]$ and, as $\ell\leq y$, also $y^t\leq y$. Moreover, since $y \not< \bar{z},$ for some $i \in [m]$ we have $y_i \geq \bar{z}_i.$ Since $y^0_i = y_i \geq \bar{z}_i$ and $y^1_i = \ell_i < \bar{z}_i,$ it follows from the intermediate value theorem the existence of $\bar{t} \in [0,1)$ such that $y^{\bar{t}}_i = \bar{z}_i.$ Then, it follows that the point $y^{\bar{t}}$ is feasible for \eqref{plzi}. As  $\ell^i$ is by Lemma \ref{lem:stronglyefficientmp} strongly efficient for this problem, we find that $\ell^i \leq y^{\bar{t}} \leq y,$ as desired.

Now, according to these three statements, we deduce that $$\Min(A(S), \R_+^m) \subseteq  P \cup \left(\Min(A(S\setminus \{\bar{z}\}), \R_+^m) \setminus D\right).$$ Thus, taking into account that $|P|\leq m|D|,\ |D|\leq m^s,\  |\Min(A(S\setminus \{\bar{z}\}), \R_+^m)| \leq m^s,$ we find that
\begin{eqnarray*}
|\Min(A(S), \R_+^m)| & \leq &  |P| + |\Min(A(S\setminus \{\bar{z}\}), \R_+^m)| - |D|  \\
                     & \leq &  m|D| + |\Min(A(S\setminus \{\bar{z}\}), \R_+^m)| - |D|\\
                     & =    &  (m - 1)|D| + |\Min(A(S\setminus \{\bar{z}\}), \R_+^m)|\\ 
                     & \leq &  (m - 1)m^s + m^s\\
                     & = & m^{s + 1},
\end{eqnarray*} and the induction step follows. This completes the proof.
\end{proof}

\begin{Remark}\label{remark:closedA}
Note that by Proposition \ref{prop: cl A-K equivalent} it holds 
$$
       \bigcap_{z \in S} (\{z\} - \Int \R_{+}^m)^c \cap B 
=   (S - \Int \R_{+}^m)^c \cap B \\
=  \cl (S - \Int \R_{+}^m)^c \cap B. $$  
\end{Remark}

Now we are able to give conditions which ensure the exactness of the reformulation \eqref{mp}, which provides in the general setting   approximations only.

\begin{Theorem} Let Assumption \ref{main_assumption} hold and suppose that at least one of the following conditions is fulfilled:
\begin{enumerate}
    \item $|\Omega| < + \infty,$
    \item $|\mathcal{U}|< + \infty$ and $C = C^\alpha$ for some $\alpha > m - 1.$
    \item $\mathcal{U} = \mathcal{U}_1 \times \ldots \times \mathcal{U}_m$, where $\mathcal{U}_j \subseteq \R^{k_j}$ are such that $\sum_{j = 1}^m k_j = k$ and $$f(x,u) = \begin{pmatrix} f_1(x, u^1)\\ \vdots \\ f_m(x, u^m)\end{pmatrix},$$ with $u^j \in \R^{k_j},$ and $C = C^\alpha$ for some $\alpha > m - 1.$
\end{enumerate}  Then,
$$ \eqref{ump} \ \mbox{  satisfies \textup{(wFDVP)}.}$$
Moreover, for $(i)$  we can choose $p=|\Omega|-1$, in case $(ii)$  we can choose $p=m^{|\mathcal{U}|}$, and in $(iii)$ the choice $p = m$ is sufficient. 
\end{Theorem}
\begin{proof}
$(i)$ This is a direct consequence of Proposition \ref{prop:spu=spl} and Theorem \ref{thm:splmainresults} \ref{item:fdvpvpp}.

$(ii)$ Note that under our assumptions $|F_{\mathcal{U}}(x)| \leq |\mathcal{U}|< + \infty$ for every $x \in  \Omega.$ Therefore, applying Lemma \ref{lem:finitelocallowerbounds} and Remark \ref{remark:closedA}, we deduce that the set-valued mapping $F$ given in \eqref{eq:Fhood} satisfies 
$$\max\limits_{x\in \Omega} |\Min(F(x),\R^m_+)| \leq m^{|\mathcal{U}|} < + \infty.$$ Thus, from  Lemma \ref{lem:mop=mop'}, Theorem \ref{thm:splmainresults} \ref{item:fdvpvpp} and Proposition \ref{prop:spu=spl}, we find that for $p = m^{|\mathcal{U}|} $ 
$$\wargmin_x \eqref{mp} = \wargmin_x \eqref{vpp} = \wargmin^l \eqref{sp} = \wargmin \eqref{ump},$$ as desired.

$(iii)$ Under our assumptions, we can consider the function $f_{\mathcal{U}}: \Omega \to \R^m$ defined as 
$$f_{\mathcal{U}}(x) := \begin{pmatrix} \max\limits_{u^1 \in \mathcal{U}_1} f_1(x, u^1)\\ \vdots \\ \max\limits_{u^m \in \mathcal{U}_m}  f_m(x, u^m)\end{pmatrix}.$$ Then, it is easy to see that $f_{\mathcal{U}} (x) \in F_{\mathcal{U}}(x)$ for each $x \in \Omega$ and that 
\begin{equation}\label{eq:owc}
     \forall\; x \in \Omega: F_{\mathcal{U}}(x) - \R^m_+ = \{f_{\mathcal{U}} (x)\}  -\R^m_+.
\end{equation} Therefore, taking into account \eqref{eq:owc}, we can assume without loss of generality that $|\mathcal{U}| = 1$ and argue like in $(ii).$
\end{proof}

\section{Conclusions}

Within this paper we followed the set-based minmax approach for studying uncertain multiobjective optimization problems. This leads to set-optimization problems for which hardly solution methods exist so far. By constructing  a new suitable set-valued mapping with compact values, we found a reformulation as a set optimization problem for which the strict  lower  type binary relation can be used. This allows to formulate a multiobjective replacement problem for which the weakly efficient solutions and the $\varepsilon$-weakly efficient solutions approximate the 
robust weakly efficient solutions to a desired accuracy. In addition to that, the approximation quality can be improved iteratively as there is a hierarchy of the approximation quality due to monotonicity results.

Moreover, we found a reformulation of this multiobjective optimization problem using semi-infinite constraints, which is the desired epigraphical reformulation in the multiobjective setting 
 and generalizes single-objective epigraphical reformulations.  We have provided conditions under which this problem even delivers all 
robust weakly efficient solutions and not just an approximation of those, i.e., under which it is exact.

As the weakly efficient solutions of the problem \eqref{mp}
are good approximations, the next step is to examine how to solve  these  multiobjective semi-infinite problems. An advantage of our method is that, since we removed the set-valued structure, techniques from single-objective robust optimization can now be applied to the constraint set. A possible idea is to derive reformulations of the semi-infinite constraints by  using duality theory, like Fenchel's strong duality theorem 
\cite[Theorem 31.1]{Rockafellar1}, which needs to be combined with  
chain rules for the conjugation of  compositions of maps, as in 
\cite[Theorem 2]{HiriartUrruty2006}. Specifically under additional assumptions, like linearity of $f$ in both the decision variable $x$ and the scenarios $u$, one can expect to obtain multiobjective optimization problems with a finite number of easier to handle constraints.

As a drawback of our method, we mention the need of some global information (find $\ell b$ and $C$) of the uncertain multiobjective  optimization problem in order to derive the new compact-valued set-valued mapping. Thus, some preprocessing with global optimization techniques is needed.          

\section*{Funding}
This research is funded by the  German Federal Ministry for Economic Affairs and Climate Action (BMWK) under grant 03EI4013B.

%%%%%%%%%%%%%%%%%%%%%%
%% Acknowledgements %%
%%%%%%%%%%%%%%%%%%%%%%
%\section{Acknowledgements}
%\label{section:Acknowledgements}

%%%%%%%%%%%%%%%%%%
%% Bibliography %%
%%%%%%%%%%%%%%%%%%
\bibliographystyle{siam}
\bibliography{references}
\end{document}